\documentclass[11pt,a4paper]{article}

\usepackage[T1]{fontenc}
\usepackage[utf8]{inputenc}
\usepackage[english]{babel}
\usepackage{amsmath,amssymb,amsfonts}
\usepackage{amsthm}
\usepackage[numbers,sort&compress]{natbib}
\usepackage{hyperref}
\usepackage{tikz}
\usetikzlibrary{arrows.meta,positioning}
\usepackage{float}
\usepackage{xcolor}
\usepackage[a4paper,margin=30mm]{geometry}

\theoremstyle{plain}
\newtheorem{theorem}{Theorem}
\newtheorem{proposition}[theorem]{Proposition}
\newtheorem{lemma}[theorem]{Lemma}
\newtheorem{corollary}[theorem]{Corollary}

\theoremstyle{definition}
\newtheorem{definition}{Definition}
\newtheorem{example}{Example}

\theoremstyle{remark}
\newtheorem{remark}{Remark}

\newcommand{\R}{\mathbb{R}}
\newcommand{\overR}{\overline{\mathbb{R}}}

\newcommand{\dom}{\operatorname{dom}}

\newcommand{\slp}{\operatorname{slp}}

\newcommand{\Lip}{\operatorname{Lip}}
\newcommand{\Aff}{\mathcal{A}}
\newcommand{\LCm}{\mathcal{LC}^{-}}
\newcommand{\NCm}{\mathcal{NC}^{-}}
\newcommand{\Fnc}{\mathcal{F}_{\mathrm{nc}}}

\hypersetup{
  pdftitle={Norm-Cone Conjugation for Nonconvex Optimization: Scalar Duality and Exact Distance Penalization},
  pdfauthor={Fernando Garcia-Castano; Miguel Angel Melguizo-Padial},
  colorlinks=true,
  linkcolor=black,
  citecolor=black,
  urlcolor=black
}

\title{Norm-Cone Conjugation for Nonconvex Optimization:\\
Scalar Duality and Exact Distance Penalization}
\author{
Fernando Garc\'ia-Casta\~no\thanks{Corresponding author: \texttt{fernando.gc@ua.es}. ORCID: 0000-0002-8352-8235}
\and
Miguel \'Angel Melguizo-Padial\thanks{\texttt{ma.mp@ua.es}. ORCID: 0000-0003-0303-791X}
}
\date{}

\begin{document}
\maketitle

\begin{center}
\small Department of Mathematics, University of Alicante,\\
C. de San Vicente del Raspeig, San Vicente del Raspeig, 03690 Alicante, Spain
\end{center}

\begin{abstract}
We develop a norm-cone conjugation framework,
generated by translated norm-cones
\[
x\mapsto r-\alpha\|x-x_0\|,
\qquad \alpha\geq 0.
\]
We show that this family generates the same abstract-convex class as the
family of Lipschitz continuous concave functions,
and exploit its explicit metric structure to obtain a concrete support
geometry, a Fenchel--Moreau-type biconjugation theorem, and an associated
norm-cone subdifferential.

The main optimization consequence is a perturbation-duality theory
based on partial norm-cone conjugation.
Although the partial conjugate depends on a slope and a centre, the centre
can be fixed at the nominal perturbation in the complete biconjugate
expression without changing its value. Consequently, the effective dual
problem is a scalar program whose only dual variable is the real slope
$\alpha\geq0$, even when the primal and perturbation
spaces are vector-valued or infinite-dimensional.
The resulting dual value is the norm-cone biconjugate of the value function
at the nominal perturbation, so that strong duality is characterized by
exact pointwise biconjugation.

For nonconvex conic inequality problems, exact distance penalization is
characterized by norm-cone supportability of the natural value function:
the exact penalty parameters form its norm-cone subdifferential, and the
least exact parameter is the minimal support slope.
Global error bounds, equivalently global metric subregularity,
provide a variational-analytic sufficient condition for exactness
and hence for strong norm-cone duality.

An infinite-dimensional example shows that the results remain
applicable with a nonconvex objective and a nonsolid ordering cone.
\end{abstract}



\section{Introduction}\label{sec1}

Convex duality is one of the central principles of modern optimization and
variational analysis. In the classical Fenchel--Moreau framework, duality is
generated by affine minorants: the conjugate of a function is defined through
continuous linear functionals, and convexity together with lower
semicontinuity ensures exact biconjugation. This representation principle is a
basic tool in convex optimization, perturbation duality and variational
analysis; see, for instance, \cite{Moreau1965,EkelandTemam1976,Rockafellar1974,Zalinescu2002,Bot2010}.
In nonconvex optimization, however, affine supports may fail to capture
the global geometry of the problem, and alternative supporting families may
provide useful dual bounds and exact penalization mechanisms.

Many optimization problems are not naturally convex, and this has motivated
generalized convexity and generalized conjugation schemes that replace the
affine structure of classical Fenchel theory by broader families of elementary
functions. A central instance is abstract convexity, originating in the works
of Kutateladze and Rubinov \cite{KutateladzeRubinov1972,KutateladzeRubinov1976} and developed further in
\cite{Balder1977,DoleckiKurcyusz1978,Rubinov2000,PallaschkeRolewicz1997,Singer2006}. In this setting, an extended-real-valued function is
represented as the upper envelope of minorants selected from a prescribed
family $H$, which may also be described through a coupling function \cite{MartinezLegaz2005}.
From the viewpoint of nonconvex optimization, such global minorants are
particularly relevant when they lead to tractable dual parameterizations or
to quantitative information on exact penalty parameters.

The purpose of the present paper is to exploit a simple metric
elementary family within this generalized-convexity framework in order to
develop a radial duality and exact-penalization scheme for nonconvex
optimization.
Following the approach to abstract convexity based on real-valued Lipschitz
continuous concave functions \cite{Gorokhovik2022}, we consider translated norm-cones
\[
x\mapsto r-\alpha\|x-x_0\|,
\qquad
r\in\mathbb R,\quad \alpha\geq0,\quad x_0\in X.
\]
These functions constitute a geometrically explicit subclass of the
real-valued Lipschitz continuous concave functions. We prove that this
restriction entails no loss at the level of abstract convexity: the functions
representable as upper envelopes of translated norm-cone minorants are
precisely those representable as upper envelopes of Lipschitz continuous
concave minorants. Thus the purpose of the norm-cone family is not to enlarge
the underlying class of abstract convex functions. Its advantage is
structural: the parameterization by centres and slopes provides a concrete
metric support geometry that is not visible at the level of a general
Lipschitz-concave elementary family.

At the foundational level, this metric structure yields a specific
conjugation theory adapted to the subsequent optimization applications.
We introduce the norm-cone conjugate, admissible slopes and heights,
norm-cone supports, the biconjugate, and the associated norm-cone
subdifferential. The biconjugate is shown to be the upper envelope of all
norm-cone minorants. For functions admitting at least one norm-cone minorant,
exact biconjugation is characterized by lower semicontinuity, yielding a
Fenchel--Moreau-type representation beyond classical convexity. The framework
also leads to explicit conjugation formulae for several nonsmooth and
nonconvex functions.

The first main optimization contribution is a perturbation-duality
theory based on partial norm-cone conjugation in the perturbation variable.
The perturbation approach is classical in conjugate duality
\cite{Bot2010,EkelandTemam1976,Rockafellar1974,Zalinescu2002},  but the norm-cone structure produces a substantial reduction
in the effective dual parameterization. Although the partial conjugate
initially depends on a slope and a centre, the centre can be fixed at the
nominal perturbation without changing the dual value. The resulting dual
problem is therefore a scalar program whose only effective dual variable is
the real slope $\alpha\geq0$, even when the primal and perturbation spaces
are vector-valued and possibly infinite-dimensional. Moreover, the optimal
dual value is exactly the norm-cone biconjugate of the value function at the
nominal perturbation. Consequently, strong norm-cone duality is equivalent
to pointwise exact biconjugation of the value function, while uniform lower
Lipschitz estimates provide a direct sufficient condition through norm-cone
supportability.
Thus the dimension of the perturbation space does not enlarge the
effective dual parameterization, although evaluating the scalar dual
objective may still require solving a nonconvex optimization problem.

The second main optimization contribution concerns exact penalization
for nonconvex conic inequality problems
\[
g(x)\in-Z_+.
\]
The natural constraint perturbation may fail to satisfy a uniform lower
Lipschitz estimate, which motivates a distance-penalized perturbation.
Exact distance penalization is characterized by norm-cone supportability of
the natural value function: the exact penalty parameters are precisely the
elements of
\[
\partial^{\sharp}v_\Phi(0_Z),
\]
and the least exact parameter is the minimal norm-cone support slope
\[
|\partial^{\sharp}v_\Phi|(0_Z).
\]
Global error bounds, equivalently global metric subregularity of the
constraint mapping, yield verifiable sufficient conditions for exactness.
Exact penalization then gives strong norm-cone duality for both the direct
and the penalized perturbation schemes.

In particular, the framework does more than provide a zero-duality-gap
criterion: it identifies the complete interval of exact distance-penalty
parameters and its least element through the metric support geometry of the
value function.

The framework is related to gauge duality and polar convolution
\cite{FriedlanderMacedoPong:2014,FriedlanderMacedoPong:2019},
generalized \(c\)-conjugation and evenly convex analysis
\cite{Fajardo:2022evenly,FajardoVidalNunez:2025},
weak conjugacy and duality in nonconvex optimization
\cite{DincYalcinKasimbeyli2020},
and nonlinear Fenchel conjugation
\cite{SchielaHerzogBergmann2024}.
In the weak-conjugacy framework, the elementary functions combine a
continuous linear component with a norm term. The norm-cone construction
studied here isolates the purely metric component of this geometry. As shown
in Section~\ref{sec:norm-cone_conjugation}, on the class
\(\mathcal F_{\rm nc}(X)\) this restriction preserves the resulting
biconjugate, while exposing a radial structure in terms of centres,
admissible slopes, and minimal support slopes. This structure is then used to
develop the distance-penalization results of
Section~\ref{sec:distance_penalization_norm_cone_duality}.

This preservation result concerns the biconjugate and, consequently, the
optimal dual value. It does not identify the first conjugates, the dual
parameterizations, or their sets of optimal solutions, which may differ in
the two frameworks.
Accordingly, the value of the radial specialization in the present
paper lies not in reproducing the full weak-conjugacy parameterization, but
in the scalar dual reduction and in the quantitative interpretation of its
slope as an exact-penalty parameter.

The paper is organized as follows.
Section~\ref{sec:preliminaries} fixes the notation.
Section~\ref{sec:normcone_abstract_convexity} places norm-cone functions
within abstract convexity.
Section~\ref{sec:norm-cone_conjugation} develops norm-cone conjugation and
biconjugation.
Section~\ref{sec:norm_cone_subdifferential} studies the norm-cone
subdifferential.
Section~\ref{sec:dual_norm_cone_optimization} develops the partial norm-cone
perturbation duality theory.
Section~\ref{sec:distance_penalization_norm_cone_duality} applies the theory
to conic inequality problems through distance penalization, error bounds and
metric subregularity.
Section~\ref{sec:conclusions_future_research} collects the main conclusions
and indicates some directions for future research.

\section{Preliminaries and Notation}
\label{sec:preliminaries}

Throughout the paper, \(X\) and \(Z\) denote real normed spaces, both norms
being denoted by \(\|\cdot\|\). The context will always make clear which norm
is meant. For \(x\in X\) and \(\delta>0\), we write
\(B_X(x,\delta):=\{u\in X:\|u-x\|\leq \delta\}\), and use the analogous
notation in \(Z\).

We set
\[
\overline{\mathbb R}:=\mathbb R\cup\{+\infty,-\infty\},
\qquad
\mathbb R_+:=[0,+\infty),
\qquad
t_+:=\max\{t,0\}.
\]
For a set \(A\), we denote its closure and interior by
\(\operatorname{cl}A\) and \(\operatorname{int}A\), respectively, and write
\(-A:=\{-a:a\in A\}\).

For a function \(f:X\to\overline{\mathbb R}\), its effective domain and
epigraph are
\[
\operatorname{dom}f:=\{x\in X:f(x)<+\infty\},
\qquad
\operatorname{epi}f:=\{(x,r)\in X\times\mathbb R:f(x)\leq r\}.
\]
The function \(f\) is proper if \(\operatorname{dom}f\neq\varnothing\) and
\(f(x)>-\infty\) for every \(x\in X\). Unless explicitly stated otherwise,
all extended-real-valued functions considered in the paper are assumed to be
proper. A function originally defined on a nonempty set \(C\subset X\) is
identified with its standard extension to \(X\), obtained by setting
\(f(x)=+\infty\) for \(x\notin C\).

A function \(f:X\to\overline{\mathbb R}\) is lower semicontinuous if
\(\operatorname{epi}f\) is closed. A real-valued function \(f:C\to\mathbb R\), $\varnothing \not =C\subset X$,
is Lipschitz continuous if there exists \(L\geq0\) such that
\(|f(x)-f(y)|\leq L\|x-y\|\) for all \(x,y\in C\), and we set
\[
\operatorname{Lip}(f)
:=
\inf\left\{
L\ge 0:
|f(x)-f(y)|\le L\|x-y\|
\quad \forall\,x,y\in C
\right\}.
\]

Thus \(f\) is Lipschitz continuous if and only if
\(\operatorname{Lip}(f)<+\infty\).

For a nonempty set \(A\subset Z\) and \(z\in Z\), we write
\(d_A(z):=\inf_{a\in A}\|z-a\|\) for the distance from \(z\) to \(A\). We also use
\(\operatorname{dist}(z,A)=d_A(z)\).

A nonempty set \(K\subset Z\) is a cone if
\(\lambda K\subset K\) for every \(\lambda\in\mathbb R_+\).
It is nontrivial if
\(\{0_Z\}\subsetneq K\subsetneq Z\), pointed if
\(K\cap(-K)=\{0_Z\}\), and solid if \(\operatorname{int}K\neq\varnothing\).
Unless explicitly stated otherwise, all cones considered below are assumed
to be nontrivial.

We use the standard order and arithmetic conventions on
\(\overline{\mathbb R}\), with
\(\sup\varnothing=-\infty\) and \(\inf\varnothing=+\infty\).
Indeterminate expressions such as \(+\infty-\infty\) are never used.

\section{Norm-cone Functions within Abstract Convexity}
\label{sec:normcone_abstract_convexity}
This section places the norm-cone family within the abstract-convexity
framework needed later in the paper. We recall only the notions required
from the \(\LCm\)-theory of Gorokhovik \cite{Gorokhovik2022} and then show
that the family of norm-cone functions generates the same class of abstract
convex functions as \(\LCm\).

We begin by recalling the relevant notions from \cite{Gorokhovik2022}. Let
\(H\) be a family of functions \(h:X\to\R\), and let \(f:X\to\overR\).
The elements of \(S^{-}(H,f):=\{h\in H:h\le f\}\) are called
\(H\)-minorants of \(f\).

The function \(f\) is said to be \(H\)-convex if
\[
S^{-}(H,f)\neq\varnothing
\]
and
\[
f(x)=\sup\{h(x):h\in S^{-}(H,f)\}
\qquad \forall x\in X.
\]
Following \cite{Gorokhovik2022}, let \(\LCm:=\LCm(X,\R)\) denote the class
of all real-valued Lipschitz continuous concave functions on \(X\). We shall
use the following characterization proved in
\cite[Theorem~3.2]{Gorokhovik2022}: a function
\(f:X\to\overline{\mathbb R}\) is \(\LCm\)-convex if and only if it is
lower semicontinuous and bounded from below by a Lipschitz continuous
function.

We now introduce the special family of elementary functions that will generate the norm-cone conjugation developed in the subsequent sections.

\begin{definition}
\label{def:normcone_functions}
Let \((X,\|\cdot\|)\) be a normed space.

For each \(\alpha\in\R_+\), define
\(\phi_\alpha(x):=-\alpha\|x\|\), \(x\in X\).

A function \(h:X\to\R\) is called a \emph{norm-cone function} if there exist \(r\in\R\), \(x_0\in X\), and \(\alpha\in\R_+\) such that
\[
h(x)
=
r+\phi_\alpha(x-x_0)
=
r-\alpha\|x-x_0\|
\qquad \forall x\in X.
\]

The collection of all norm-cone functions on \(X\) is denoted by
\(\NCm:=\NCm(X,\R)\).
\end{definition}

\begin{remark}
\label{rem:normcone_subset_lcm}
Every norm-cone function \(h\in\NCm\) is concave and Lipschitz continuous. More precisely, if
\(
h(x)=r-\alpha\|x-x_0\|,
\)
then
\(
\Lip(h)=\alpha.
\)
Consequently, \(\NCm\subset\LCm\).
\end{remark}

The next result shows that the family \(\NCm\) generates, through upper envelopes, exactly the same class of functions as \(\LCm\). Thus, from the point of view of abstract convexity, no generality is lost by replacing the whole class of Lipschitz continuous concave functions by the much smaller and geometrically explicit family of norm-cone functions.

\begin{theorem}
\label{thm:normcone_lcm_equivalence}
For a function \(f:X\to\overR\), the following assertions are equivalent:
\begin{enumerate}
\item \(f\) is \(\NCm\)-convex;

\item \(f\) is \(\LCm\)-convex;

\item \(f\) is lower semicontinuous and bounded from below by a norm-cone function.
\end{enumerate}
\end{theorem}

\begin{proof}
\(1)\Rightarrow 2)\).
Since \(\NCm\subset\LCm\), one has
\(
S^{-}(\NCm,f)\subset S^{-}(\LCm,f).
\)
Hence
\[
f(x)
=
\sup\{h(x):h\in S^{-}(\NCm,f)\}
\le
\sup\{h(x):h\in S^{-}(\LCm,f)\}
\le
f(x),
\]
which proves that \(f\) is \(\LCm\)-convex.

\medskip

\(2)\Rightarrow 3)\).
By \cite[Theorem~3.2]{Gorokhovik2022}, \(f\) is lower semicontinuous and bounded from below by a Lipschitz continuous function \(g:X\to\R\).
Fix \(x_0\in X\) and set \(L:=\Lip(g)\). Then
\[
g(x)\ge g(x_0)-L\|x-x_0\|
\qquad \forall x\in X.
\]
Therefore, the norm-cone function
\(
h(x):=g(x_0)-L\|x-x_0\|
\)
satisfies
\(
h\le g\le f.
\)

\medskip

\(3)\Rightarrow 1)\).
Assume that \(f\) is lower semicontinuous and bounded from below by a
norm-cone function \(h_0\in\NCm\). Since \(h_0\in\LCm\),  \cite[Theorem~3.2]{Gorokhovik2022}
implies that \(f\) is \(\LCm\)-convex. Hence
\[
f(x)
=
\sup\{g(x):g\in S^{-}(\LCm,f)\},
\qquad \forall x\in X.
\]

We show that every \(g\in S^{-}(\LCm,f)\) is itself the upper envelope of
norm-cone minorants. Let \(g\in S^{-}(\LCm,f)\) and let \(L:=\Lip(g)\). For each
\(y\in X\), define
\[
h_y(x):=g(y)-L\|x-y\|,
\qquad x\in X.
\]
Then \(h_y\in\NCm\). Moreover, since \(g\) is \(L\)-Lipschitz,
\[
g(x)\ge g(y)-L\|x-y\|=h_y(x),
\qquad \forall x\in X.
\]
Thus \(h_y\in S^{-}(\NCm,g)\subset S^{-}(\NCm,f)\). On the other hand,
\(h_x(x)=g(x)\). Therefore
\[
g(x)
=
\sup_{y\in X} h_y(x)
\le
\sup\{h(x):h\in S^{-}(\NCm,f)\},
\qquad \forall x\in X.
\]

Taking the supremum over all \(g\in S^{-}(\LCm,f)\), we get
\[
f(x)
\le
\sup\{h(x):h\in S^{-}(\NCm,f)\},
\qquad \forall x\in X.
\]
The reverse inequality follows from \(S^{-}(\NCm,f)\subset S^{-}(\LCm,f)\).
Hence
\[
f(x)
=
\sup\{h(x):h\in S^{-}(\NCm,f)\},
\qquad \forall x\in X,
\]
and consequently \(f\) is \(\NCm\)-convex.
\end{proof}

\begin{remark}\label{rem:fenchel_as_abstract_convexity}
The previous theorem places norm-cone convexity within the general framework of
abstract convexity generated by families of support functions.

Recall that the classical Fenchel theory is generated by the family
\[
\Aff(X,\R)
\]
of continuous affine functions on $X$. In this setting, lower semicontinuous convex
functions are precisely the abstract convex functions generated by
$\Aff(X,\R)$.

Since every continuous affine function is globally Lipschitz and both convex and
concave, one has
\[
\Aff(X,\R)
\subset
\LCm(X,\R).
\]
Consequently, the class of $\LCm$-convex functions --- and therefore, by
Theorem~\ref{thm:normcone_lcm_equivalence}, the class of $\NCm$-convex
functions --- extends the classical Fenchel framework beyond convexity.
\end{remark}

\section{Norm-cone Conjugation: Definition and Properties}
\label{sec:norm-cone_conjugation}
We now develop the conjugation scheme generated by norm-cone minorants.
After introducing the norm-cone conjugate and biconjugate, admissible
slopes, and the class \(\mathcal F_{\rm nc}(X)\), we interpret the
conjugate through admissible heights and norm-cone supports. We then
derive a geometric representation of the biconjugate as the upper
envelope of all norm-cone minorants and obtain the corresponding
pointwise and global exactness results. The section concludes with
several examples illustrating explicit norm-cone conjugation formulae
for nonsmooth and nonconvex functions.

\subsection{Norm-cone Conjugation and Admissible Slopes}

\begin{definition}\label{def:metric_conjugates}
Let $X$ be a normed space and let $f:X\to\overR$ be a proper function.

\begin{enumerate}
\item The norm-cone conjugate function is
\[
f^\sharp:\R_+\times X\to\overR,
\qquad
f^\sharp(\alpha,x):=
\sup_{x'\in X}
\bigl(
\phi_\alpha(x'-x)-f(x')
\bigr).
\]
\item The norm-cone biconjugate function is
\[
f^{\sharp\sharp}:X\to\overR,
\qquad
f^{\sharp\sharp}(x):=
\sup_{(\alpha,x')\in\R_+\times X}
\bigl(
\phi_\alpha(x-x')-f^\sharp(\alpha,x')
\bigr).
\]
\end{enumerate}
\end{definition}

We now introduce the class of slopes compatible with the asymptotic behaviour of a
function relative to the family of norm-cone functions.

\begin{definition}\label{def:set_conjugate}
Let \(X\) be a normed space and let
\(f:X\to\overR\)
be a proper function. The set of \emph{admissible slopes} of \(f\) is defined by
\[
\slp(f):=
\Bigl\{
\alpha\ge0:
\sup_{x\in X}
\bigl(\phi_\alpha(x)-f(x)\bigr)
<+\infty
\Bigr\}.
\]
\end{definition}

Admissible slopes characterize the functions that admit a global minorization by a translate of a norm-cone function.

\begin{definition} \label{def:norm_cone_minorized_functions}
Let \(X\) be a normed space. We denote by
\[
\mathcal{F}_{\mathrm{nc}}(X)
\]
the class of all proper functions
\(f:X\to\overR\)
such that
\(
\slp(f)\neq\varnothing.
\)
Functions in \(\mathcal{F}_{\mathrm{nc}}(X)\) are called
\emph{norm-cone minorized functions}.
\end{definition}

The finiteness of the conjugate for a fixed slope does not depend on the
choice of centre. Indeed, for every \(\alpha\geq0\) and \(x\in X\), the
triangle inequality gives

\[
\phi_\alpha(z-x)
\leq
\phi_\alpha(z)+\alpha\|x\|,
\qquad
\phi_\alpha(z)
\leq
\phi_\alpha(z-x)+\alpha\|x\|,
\]

and therefore

\[
\bigl|
f^\sharp(\alpha,x)-f^\sharp(\alpha,0_X)
\bigr|
\leq
\alpha\|x\|,
\]

whenever one, and hence both, of these quantities is finite.
Consequently, for every \(\alpha\geq0\), the following conditions are
equivalent:

\[
\alpha\in\operatorname{slp}(f),
\qquad
f^\sharp(\alpha,x)<+\infty
\ \text{for some }x\in X,
\qquad
f^\sharp(\alpha,x)<+\infty
\ \text{for every }x\in X.
\]

In particular, if \(\alpha\notin\operatorname{slp}(f)\), then

\[
f^\sharp(\alpha,x)=+\infty,
\qquad x\in X.
\]

The next result characterizes the class \(\mathcal F_{\rm nc}(X)\) through the existence of global norm-cone minorants, or equivalently, global lower bounds of norm-cone type.

\begin{proposition}\label{prop:nc_minorized}
Let \(X\) be a normed space and let \(f:X\to\overR\) be proper. The following
assertions are equivalent:
\begin{enumerate}
\item \(f\in\Fnc(X)\);
\item \(S^{-}(\NCm,f)\neq\varnothing\);
\item there exist \(r\in\R\) and \(\alpha\ge0\) such that
\[
f(x)\ge r-\alpha\|x\|,
\qquad x\in X.
\]
\end{enumerate}
\end{proposition}

\begin{proof}
\(1)\Rightarrow3)\).
Choose \(\alpha\in\slp(f)\) and set
\[
M:=\sup_{x\in X}\bigl(\phi_\alpha(x)-f(x)\bigr)<+\infty.
\]
Then
\[
f(x)\ge -M+\phi_\alpha(x)
      =-M-\alpha\|x\|,
\qquad x\in X.
\]

\(3)\Rightarrow2)\) is immediate, since
\(x\mapsto r-\alpha\|x\|\) belongs to \(\NCm\) and is a minorant of \(f\).

Finally, assume \(2)\). Then for some \(r_0\in\R\), \(\alpha\ge0\), and
\(x_0\in X\),
\[
r_0-\alpha\|x-x_0\|\le f(x),
\qquad x\in X.
\]
Since
\[
-\alpha\|x-x_0\|
\ge -\alpha\|x\|-\alpha\|x_0\|,
\]
we obtain
\[
f(x)\ge
\bigl(r_0-\alpha\|x_0\|\bigr)-\alpha\|x\|,
\qquad x\in X.
\]
Thus \(3)\) holds.
\end{proof}

Thus \(\Fnc(X)\) records only the existence of a norm-cone minorant, whereas
\(\NCm\)-convexity requires the exact recovery of the function as the upper
envelope of all its norm-cone minorants.

\begin{example}\label{ex:funciones_norm_cone_minorized}
The class \(\Fnc(X)\) contains all proper functions bounded from below and all
proper Lipschitz continuous functions. Moreover, every proper lower
semicontinuous convex function belongs to \(\Fnc(X)\). In contrast,
\(g(x)=x^{2n+1}\notin\Fnc(\R)\) for every \(n\in\mathbb N\).
\end{example}

\begin{proof}
Only the convex case requires an argument. Let
\(f:X\to\overR\) be proper, lower semicontinuous and convex, and fix
\(x_0\in\dom f\). There exists \(\delta>0\) such that
\[
f(u)\ge f(x_0)-1
\qquad\text{whenever }\|u-x_0\|\le\delta.
\]
For \(x\in\dom f\), \(x\neq x_0\), set
\[
\lambda_x:=
\min\left\{1,\frac{\delta}{2\|x-x_0\|}\right\}.
\]
Convexity applied to
\(u_x=(1-\lambda_x)x_0+\lambda_x x\) gives
\[
f(x)\ge f(x_0)-\frac1{\lambda_x}.
\]
Since
\[
\frac1{\lambda_x}
\le
1+\frac{2}{\delta}\|x-x_0\|
\le
1+\frac{2}{\delta}\|x_0\|
+\frac{2}{\delta}\|x\|,
\]
we obtain
\[
f(x)\ge
f(x_0)-1-\frac{2}{\delta}\|x_0\|
-\frac{2}{\delta}\|x\|,
\qquad x\in X.
\]
Hence \(f\in\Fnc(X)\) by Proposition~\ref{prop:nc_minorized}. The remaining positive assertions are immediate from the definition. Finally,
for every \(\alpha\geq0\),

\[
x^{2n+1}+\alpha|x|
\longrightarrow -\infty,
\qquad\text{as }x\to-\infty.
\]

Hence there is no \(r\in\mathbb R\) and \(\alpha\geq0\) such that

\[
x^{2n+1}\geq r-\alpha|x|,
\qquad x\in\mathbb R.
\]

Therefore \(x^{2n+1}\notin\mathcal F_{\rm nc}(\mathbb R)\) by
Proposition~\ref{prop:nc_minorized}.

\end{proof}

The norm-cone conjugation enjoys several elementary properties analogous to those
of the classical Fenchel conjugation.

\begin{proposition}\label{prop:basic_properties_metric}
Let \(X\) be a normed space and let
\(f,g\in\mathcal{F}_{\mathrm{nc}}(X)\). Then:

\begin{enumerate}
\item For every \(\alpha\in\slp(f)\), the function
\(x\longrightarrow f^\sharp(\alpha,x)\) is \(\alpha\)-Lipschitz on \(X\).

\item \(f^{\sharp\sharp}\le f\) on \(X\).

\item If \(f\le g\), then
\[
\slp(f)\subset\slp(g),
\qquad
f^\sharp\ge g^\sharp,
\qquad
f^{\sharp\sharp}\le g^{\sharp\sharp}.
\]

\item For every \(c\in\mathbb R\),
\[
\slp(f+c)=\slp(f),
\qquad
(f+c)^\sharp=f^\sharp-c,
\qquad
(f+c)^{\sharp\sharp}=f^{\sharp\sharp}+c.
\]
\end{enumerate}
\end{proposition}

\begin{proof}
For \(x,y,z\in X\), the triangle inequality gives
\[
\phi_\alpha(z-x)
\le
\phi_\alpha(z-y)+\alpha\|x-y\|.
\]
Taking the supremum over \(z\) yields
\[
f^\sharp(\alpha,x)
\le
f^\sharp(\alpha,y)+\alpha\|x-y\|,
\]
and exchanging \(x\) and \(y\) proves (1).

For (2), the definition of \(f^\sharp\) gives, for every
\((\alpha,x')\in\R_+\times X\),
\[
f^\sharp(\alpha,x')
\ge
\phi_\alpha(x-x')-f(x),
\]
hence
\[
\phi_\alpha(x-x')-f^\sharp(\alpha,x')\le f(x).
\]
Taking the supremum over \((\alpha,x')\) gives
\(f^{\sharp\sharp}\le f\).

If \(f\le g\), then directly from the definitions
\[
\slp(f)\subset\slp(g)
\qquad\text{and}\qquad
f^\sharp\ge g^\sharp,
\]
which in turn gives \(f^{\sharp\sharp}\le g^{\sharp\sharp}\). This proves (3).
Finally, (4) follows immediately from the definitions.
\end{proof}

We close this subsection with the basic inequality associated with the
norm-cone conjugation. It is the analogue of the Fenchel--Young inequality
in the present setting.

\begin{proposition}
\label{prop:fenchel_young_metric}
Let \(X\) be a normed space and let \(f\in\mathcal F_{\mathrm{nc}}(X)\). Then
\[
-f^\sharp(\alpha,x)+\phi_\alpha(z-x)\le f(z),
\qquad
\forall (\alpha,x)\in\mathbb R_+\times X,\ \forall z\in X.
\]
Equivalently,
\[
f^\sharp(\alpha,x)+f(z)-\phi_\alpha(z-x)\ge0,
\qquad
\forall (\alpha,x)\in\mathbb R_+\times X,\ \forall z\in X.
\]

\end{proposition}

\begin{proof}
By definition,
\[
f^\sharp(\alpha,x)
=
\sup_{y\in X}
\bigl(\phi_\alpha(y-x)-f(y)\bigr).
\]
Evaluating the supremum at \(y=z\) gives
\[
f^\sharp(\alpha,x)\ge \phi_\alpha(z-x)-f(z).
\]
Equivalently,
\[
-f^\sharp(\alpha,x)+\phi_\alpha(z-x)\le f(z),
\]
which proves the claim.
\end{proof}

Thus, for each \(\alpha\in\operatorname{slp}(f)\) and \(x\in X\),
the quantity
\[
-f^\sharp(\alpha,x)\in\mathbb R
\]
is the largest vertical shift associated with the slope \(\alpha\)
and centre \(x\) for which the corresponding norm-cone function
remains below \(f\). For \(\alpha\notin\operatorname{slp}(f)\), one has
\(f^\sharp(\alpha,x)=+\infty\) for every \(x\in X\), so no real maximal
admissible height is associated with such a slope. This observation
motivates the notion of admissible heights introduced in the next
subsection.

\subsection{Admissible Heights and Norm-cone Supports}

The conjugate \(f^\sharp\) determines the maximal vertical shifts for which a
norm-cone function remains below \(f\). This leads to the following notion.

\begin{definition}\label{def:admissible_heights}
Let \(X\) be a normed space and let \(f\in\mathcal F_{\mathrm{nc}}(X)\).
For \(\alpha\in\slp(f)\) and \(x\in X\), define the set of
\emph{admissible heights} by
\[
\mathcal H_f(\alpha,x)
:=
\bigl\{
r\in\mathbb R:
r+\phi_\alpha(z-x)\le f(z)\ \text{for all } z\in X
\bigr\}.
\]
Each function \(z\mapsto r+\phi_\alpha(z-x)\), with
\(r\in\mathcal H_f(\alpha,x)\), is called a \emph{norm-cone minorant} of \(f\).
If, in addition, it touches \(f\) at some point \(\bar z\in X\), that is,
\[
r+\phi_\alpha(\bar z-x)=f(\bar z),
\]
we call it a \emph{norm-cone support} of \(f\) at \(\bar z\in\dom f\).
\end{definition}

The admissible heights are completely determined by the conjugate.

\begin{lemma}\label{lem:height_structure}
Let \(X\) be a normed space and let \(f\in\Fnc(X)\). For every
\(\alpha\in\slp(f)\) and \(x\in X\),
\[
\mathcal H_f(\alpha,x)
=
(-\infty,-f^\sharp(\alpha,x)].
\]
\end{lemma}

\begin{proof}
Since \(\alpha\in\slp(f)\), \(f^\sharp(\alpha,x)\in\R\). Moreover,
\[
\begin{aligned}
r\in\mathcal H_f(\alpha,x)
&\Longleftrightarrow
r+\phi_\alpha(z-x)\le f(z),
\quad\forall z\in X,\\
&\Longleftrightarrow
r\le
\inf_{z\in X}\bigl(f(z)-\phi_\alpha(z-x)\bigr),\\
&\Longleftrightarrow
r\le -f^\sharp(\alpha,x).
\end{aligned}
\]
\end{proof}

The particular metric structure of the norm-cone kernels allows the centre in
the biconjugation formula to be fixed at the evaluation point.

\begin{proposition}
\label{prop:fixed_center_biconjugate}
Let \(X\) be a normed space and let \(f\in\Fnc(X)\). Then, for every \(x\in X\),
\[
f^{\sharp\sharp}(x)
=
\sup_{\alpha\in\slp(f)}
\bigl(-f^\sharp(\alpha,x)\bigr).
\]
\end{proposition}

\begin{proof}
Fix \(x\in X\) and \(\alpha\in\slp(f)\). For every \(x'\in X\), the triangle
inequality gives
\[
\phi_\alpha(z-x')
\ge
\phi_\alpha(z-x)+\phi_\alpha(x-x'),
\qquad z\in X.
\]
Hence
\[
f^\sharp(\alpha,x')
\ge
f^\sharp(\alpha,x)+\phi_\alpha(x-x'),
\]
and therefore
\[
\phi_\alpha(x-x')-f^\sharp(\alpha,x')
\le
-f^\sharp(\alpha,x).
\]
Equality holds for \(x'=x\). Consequently,
\[
\sup_{x'\in X}
\bigl(
\phi_\alpha(x-x')-f^\sharp(\alpha,x')
\bigr)
=
-f^\sharp(\alpha,x).
\]
If \(\alpha\notin\operatorname{slp}(f)\), then, by the observation
following Definition~\ref{def:norm_cone_minorized_functions},
\(f^\sharp(\alpha,x_0)=+\infty\) for every \(x_0\in X\), so such
slopes do not contribute to the biconjugate. Taking the supremum over
\(\alpha\in\slp(f)\) proves the formula.
\end{proof}

\begin{remark}
\label{rem:weak_conjugacy_biconjugate}
The norm-cone construction is closely related to the weak conjugacy used in
nonconvex optimization; see \cite{DincYalcinKasimbeyli2020} and the
references therein. Let \(X^*\) denote the continuous dual of \(X\).
The weak conjugate is defined by
\[
f^w(x^*,\alpha,x_0)
:=
\sup_{y\in X}
\bigl(
\langle x^*,y\rangle
-\alpha\|y-x_0\|
+\alpha\|x_0\|
-f(y)
\bigr),
\]
where \(x^*\in X^*\), \(\alpha\ge0\), and \(x_0\in X\), and the corresponding
weak biconjugate is
\[
f^{ww}(x)
:=
\sup_{(x^*,\alpha)\in X^*\times\R_+}
\bigl(
\langle x^*,x\rangle+\alpha\|x\|
-f^w(x^*,\alpha,x)
\bigr).
\]

Setting \(x^*=0\) gives
\[
f^w(0,\alpha,x_0)
=
f^\sharp(\alpha,x_0)+\alpha\|x_0\|,
\]
so, after accounting for the centre-dependent normalization term
\(\alpha\|x_0\|\) used in the definition of the weak conjugate, the
norm-cone conjugation corresponds to its purely radial specialization.

If \(\mathcal H_w\) denotes the elementary family underlying weak conjugacy,
including vertical translations, then
\[
\NCm\subset\mathcal H_w\subset\LCm.
\]
Hence, by Theorem~\ref{thm:normcone_lcm_equivalence}, these three elementary
families generate the same class of abstract convex functions.

More specifically, if \(f\in\Fnc(X)\), then
\[
f^{ww}=f^{\sharp\sharp}.
\]
Indeed, for every \(x^*\in X^*\), \(\alpha\ge0\), and \(x\in X\),
\[
\begin{aligned}
\langle x^*,x\rangle+\alpha\|x\|-f^w(x^*,\alpha,x)
&=
\inf_{y\in X}
\bigl(
f(y)+\langle x^*,x-y\rangle
+\alpha\|y-x\|
\bigr)\\
&\le
\inf_{y\in X}
\bigl(
f(y)+(\alpha+\|x^*\|)\|y-x\|
\bigr)\\
&=
-f^\sharp(\alpha+\|x^*\|,x)
\le f^{\sharp\sharp}(x).
\end{aligned}
\]
Taking the supremum over \((x^*,\alpha)\) gives
\(f^{ww}(x)\le f^{\sharp\sharp}(x)\). Conversely, restricting the supremum
defining \(f^{ww}\) to \(x^*=0\) and using
Proposition~\ref{prop:fixed_center_biconjugate} yields
\[
f^{ww}(x)
\ge
\sup_{\alpha\ge0}\bigl(-f^\sharp(\alpha,x)\bigr)
=
f^{\sharp\sharp}(x).
\]
Thus the continuous linear component enriches the parametrization of the
first weak conjugate, but it neither enlarges the abstract-convex class
generated by the elementary functions nor, on \(\Fnc(X)\), the resulting
biconjugate.
\end{remark}

Combining Lemma~\ref{lem:height_structure} with
Proposition~\ref{prop:fixed_center_biconjugate} also gives
\[
f^{\sharp\sharp}(x)
=
\sup
\bigl\{
r\in\R:
\exists\,\alpha\in\slp(f)
\text{ with }
r\in\mathcal H_f(\alpha,x)
\bigr\}.
\]

\subsection{Geometric Interpretation of Norm-cone Supports}

We finally make explicit the geometry encoded by the previous constructions. The
elementary objects are the finite-valued functions obtained by translating and
vertically shifting the kernels \(\phi_\alpha\).

\begin{definition}\label{def:norm_cone_parametrization}
For \(x\in X\), \(\alpha\ge0\), and \(r\in\mathbb R\), we denote by
\[
K_{x,\alpha,r}(z):=
r-\alpha\|z-x\|
=
r+\phi_\alpha(z-x),
\qquad z\in X,
\]
the corresponding element of \(\NCm\). We refer to \(x\), \(\alpha\), and \(r\)
as its centre, slope, and height, respectively.
\end{definition}

Geometrically, \(K_{x,\alpha,r}\) is a downward norm-generated surface with
vertex \((x,r)\). Its epigraph is determined by the distance to the centre \(x\),
and the parameter \(\alpha\) controls its aperture. These functions play, in the
present framework, the role played by affine minorants in classical convex
analysis. For \(f\in\Fnc(X)\), let
\[
\mathcal K_f
:=
\bigl\{
(y,\alpha,r)\in X\times\R_+\times\R:
K_{y,\alpha,r}\le f
\bigr\}.
\]
Every \((y,\alpha,r)\in\mathcal K_f\) satisfies
\(\alpha\in\slp(f)\), since
\[
K_{y,\alpha,r}(z)
\ge
r-\alpha\|y\|-\alpha\|z\|,
\qquad z\in X.
\]
The next result identifies \(f^{\sharp\sharp}\) with the upper envelope of all
norm-cone minorants of \(f\).
\begin{proposition}
\label{prop:cone_envelope_representation}
Let \(X\) be a normed space and let \(f\in\Fnc(X)\). Then, for every \(z\in X\),
\[
f^{\sharp\sharp}(z)
=
\sup_{(y,\alpha,r)\in\mathcal K_f}
K_{y,\alpha,r}(z).
\]
\end{proposition}

\begin{proof}
If \((y,\alpha,r)\in\mathcal K_f\), then
\(r\in\mathcal H_f(\alpha,y)\), and therefore, by
Lemma~\ref{lem:height_structure},
\[
K_{y,\alpha,r}(z)
\le
-f^\sharp(\alpha,y)+\phi_\alpha(z-y)
\le
f^{\sharp\sharp}(z).
\]
Thus the right-hand side does not exceed \(f^{\sharp\sharp}(z)\).

Conversely, fix \(\alpha\in\slp(f)\). By
Lemma~\ref{lem:height_structure},
\[
-f^\sharp(\alpha,z)\in\mathcal H_f(\alpha,z),
\]
and hence
\[
\bigl(z,\alpha,-f^\sharp(\alpha,z)\bigr)\in\mathcal K_f.
\]
Therefore,
\[
\sup_{(y,\beta,r)\in\mathcal K_f}K_{y,\beta,r}(z)
\ge
K_{z,\alpha,-f^\sharp(\alpha,z)}(z)
=
-f^\sharp(\alpha,z).
\]
Since this holds for every \(\alpha\in\slp(f)\), we obtain
\[
\sup_{(y,\beta,r)\in\mathcal K_f}K_{y,\beta,r}(z)
\ge
\sup_{\alpha\in\slp(f)}
\bigl(-f^\sharp(\alpha,z)\bigr)
=
f^{\sharp\sharp}(z),
\]
where the last equality follows from
Proposition~\ref{prop:fixed_center_biconjugate}.
\end{proof}

Equivalently,
\[
\operatorname{epi}(f^{\sharp\sharp})
=
\bigcap_{(y,\alpha,r)\in\mathcal K_f}
\operatorname{epi}(K_{y,\alpha,r}).
\]

Thus the biconjugate is the largest function generated from norm-cone
minorants of \(f\). In particular, every point at which \(f\) admits a
norm-cone support is a point of exact biconjugation.

\begin{corollary}
\label{cor:pointwise_exactness_norm_cone_support}
Let \(X\) be a normed space, let \(f\in\mathcal F_{\rm nc}(X)\), and let
\(x\in\dom f\). If \(f\) admits a norm-cone support at \(x\), then

\[
f^{\sharp\sharp}(x)=f(x).
\]

\end{corollary}

\begin{proof}
Let \(K_{y,\alpha,r}\) be a norm-cone support of \(f\) at \(x\). Then
\(K_{y,\alpha,r}\leq f\) and

\[
K_{y,\alpha,r}(x)=f(x).
\]

By Proposition~\ref{prop:cone_envelope_representation},

\[
f^{\sharp\sharp}(x)
\geq K_{y,\alpha,r}(x)
=f(x).
\]

On the other hand, Proposition~\ref{prop:basic_properties_metric} yields
\(f^{\sharp\sharp}\leq f\). Hence

\[
f^{\sharp\sharp}(x)=f(x).
\]

\end{proof}

Pointwise exactness at support points becomes global when norm-cone
minorants recover \(f\) everywhere. This yields the following
Fenchel--Moreau-type exactness criterion.

\begin{corollary}
\label{cor:exact_norm_cone_biconjugation}
Let \(X\) be a normed space and let \(f\in\Fnc(X)\). Then the following
assertions are equivalent:
\begin{enumerate}
\item \(f^{\sharp\sharp}=f\);
\item \(f\) is \(\NCm\)-convex;
\item \(f\) is lower semicontinuous.
\end{enumerate}
\end{corollary}

\begin{proof}
By Proposition~\ref{prop:cone_envelope_representation},
\(f^{\sharp\sharp}=f\) if and only if \(f\) is \(\NCm\)-convex.
Moreover, since \(f\in\mathcal F_{\rm nc}(X)\),
Proposition~\ref{prop:nc_minorized} guarantees that \(f\) is bounded
from below by a norm-cone function. Hence
Theorem~\ref{thm:normcone_lcm_equivalence} shows that, within
\(\mathcal F_{\rm nc}(X)\), lower semicontinuity is equivalent to
\(\NCm\)-convexity. Therefore the three assertions are equivalent.
\end{proof}

\subsection{Examples of Norm-cone Conjugates}\label{sec:examples}
The following examples illustrate how the abstract-convexity framework generated
by norm-cone functions leads to explicit conjugation formulae for concrete
metric objects.

\begin{example}
Let \(\mathbf{1}_{(0,\infty)},\mathbf{1}_{[0,\infty)}:\R\to\R\) be defined by
\[
\mathbf{1}_{(0,\infty)}(x)
:=
\begin{cases}
1, & x\in(0,\infty),\\
0, & x\notin(0,\infty),
\end{cases}
\qquad
\mathbf{1}_{[0,\infty)}(x)
:=
\begin{cases}
1, & x\in[0,\infty),\\
0, & x\notin[0,\infty).
\end{cases}
\]
Then \(\slp(\mathbf{1}_{(0,\infty)})=\slp(\mathbf{1}_{[0,\infty)})=\R_+\), and, for every \(x\in\R\),
\[
\mathbf{1}^\sharp_{(0,\infty)}(\alpha,x)
=
\mathbf{1}^\sharp_{[0,\infty)}(\alpha,x)
=
\begin{cases}
0, & x \le 0,\\[2pt]
-\alpha x, & 0<x \text{ and } 0\le \alpha \le \dfrac{1}{x},\\[6pt]
-1, & 0<x \text{ and } \alpha>\dfrac{1}{x}.
\end{cases}
\]
In contrast, the metric biconjugate distinguishes the two cases:
\[
\mathbf{1}^{\sharp\sharp}_{(0,\infty)}
=
\mathbf{1}_{(0,\infty)},
\qquad
\mathbf{1}^{\sharp\sharp}_{[0,\infty)}
=
\mathbf{1}_{(0,\infty)}
\neq
\mathbf{1}_{[0,\infty)}.
\]
\end{example}

\begin{example}
Let
\(
f(x):=-|x|.
\)
Then \(\slp(f)=[1,+\infty)\), and
\(f^\sharp(\alpha,x)=|x|\) for every \(\alpha\ge1\) and \(x\in\R\).
Consequently, \(f^{\sharp\sharp}=f\).
This example shows that admissible slopes need not coincide with \(\R_+\), and
that the asymptotic behaviour of the function determines the minimal slope of
supporting norm-cones.
\end{example}

\begin{example}
Let \(X\) be a normed space, let \(A\subset X\) be nonempty and closed, and
define \(f(x):=d_A(x)\), \(x\in X\).
Then \(f\) is \(1\)-Lipschitz and bounded from below by \(0\). Hence
\(
\slp(f)=\R_+.
\)
Moreover, for every \(\alpha\in\R_+\) and every \(x\in X\),
\[
f^\sharp(\alpha,x)
=
-\min\{\alpha,1\}\,d_A(x).
\]
In particular, \(f^\sharp(\alpha,x)=-d_A(x)\) for every \(\alpha\ge1\) and
\(x\in X\). Consequently, \(f^{\sharp\sharp}=f\).

This example shows that norm-cone conjugation applies naturally to purely
metric objects. Although \(f\) need not be smooth or convex, it admits an
exact supporting norm-cone at every point and is therefore norm-cone
supportable.
\end{example}

\section{Norm-cone Subdifferential}
\label{sec:norm_cone_subdifferential}

We now introduce a pointwise notion of support associated with the
norm-cone conjugation. In contrast with the previous section, where
biconjugation was described through all norm-cone minorants of \(f\), the
following definition singles out those norm-cones whose vertex is placed at a
given point of the graph of \(f\).

\begin{definition}
\label{def:metric_subdifferential}
Let \(X\) be a normed space and let
\(f:X\to\overR\). For every \(x\in X\) such that \(f(x)\in\R\), the
\emph{norm-cone subdifferential} of \(f\) at \(x\) is defined by
\[
\partial^\sharp f(x)
:=
\bigl\{
\alpha\ge0:
f(z)\ge f(x)+\phi_\alpha(z-x)
\ \text{for all } z\in X
\bigr\}.
\]
If \(f(x)\notin\R\), we set
\[
\partial^\sharp f(x):=\varnothing.
\]
\end{definition}

The next result records the elementary structure of the norm-cone
subdifferential at a point.

\begin{proposition}\label{prop:left_closed_subdiff}
Let \(X\) be a normed space, let
\(f:X\to\overR\) be proper, and let
\(x\in\operatorname{dom}f\). If \(\partial^\sharp f(x)\neq\varnothing\), then
\(f\) is lower semicontinuous at \(x\) and there exists
\(\alpha_x\ge0\) such that
\[
\partial^\sharp f(x)=[\alpha_x,+\infty).
\]
\end{proposition}

\begin{proof}
Let \(\alpha\in\partial^\sharp f(x)\). Then
\[
f(z)\ge f(x)-\alpha\|z-x\|,
\qquad \forall z\in X.
\]
If \(x_n\to x\), it follows that
\[
\liminf_{n\to\infty} f(x_n)
\ge
\liminf_{n\to\infty}
\bigl(f(x)-\alpha\|x_n-x\|\bigr)
=
f(x),
\]
so \(f\) is lower semicontinuous at \(x\).

Moreover, if \(\beta\ge\alpha\), then
\[
f(z)\ge f(x)-\alpha\|z-x\|
\ge
f(x)-\beta\|z-x\|,
\qquad \forall z\in X,
\]
and hence \(\beta\in\partial^\sharp f(x)\). Thus
\(\partial^\sharp f(x)\) is upward closed.

Set \(\alpha_x:=\inf\partial^\sharp f(x)\). Choose
\((\alpha_n)\subset\partial^\sharp f(x)\) with
\(\alpha_n\downarrow\alpha_x\). For every \(z\in X\),
\[
f(z)\ge f(x)-\alpha_n\|z-x\|,
\qquad \forall n.
\]
Passing to the limit gives
\(f(z)\ge f(x)-\alpha_x\|z-x\|\), so
\(\alpha_x\in\partial^\sharp f(x)\). Therefore
\(\partial^\sharp f(x)=[\alpha_x,+\infty)\).
\end{proof}

The preceding proposition allows us to associate with each point the minimal
slope of a norm-cone support whenever such a support exists.

\begin{definition}
\label{def:minimal_norm-cone_support_slope}
Let \(X\) be a normed space and let
\(f:X\to\overR\) be proper. The \emph{minimal norm-cone support slope} of \(f\) is the function
\[
|\partial^\sharp f|:X\to[0,+\infty],
\qquad
|\partial^\sharp f|(x):=
\inf\partial^\sharp f(x),
\]
with the convention \(\inf\varnothing:=+\infty\).
\end{definition}

For \(x\in\dom f\), \(|\partial^\sharp f|(x)\) is the minimal slope of a norm-cone support
with vertex at \((x,f(x))\). This is a global support quantity: it records the
least aperture needed for a norm-cone based at \(x\) to remain below the whole
function \(f\).

\begin{remark}
\label{rem:relation_gorokhovik_lcm_subdiff}
The norm-cone subdifferential is closely related to the
\(\LCm\)-subdifferentiability considered by Gorokhovik \cite{Gorokhovik2022}.
By \cite[Theorem~3.8]{Gorokhovik2022}, a function \(f\) is
\(\LCm\)-subdifferentiable at \(x\in\dom f\) if and only if there exists
\(k\ge0\) such that
\[
f(z)\ge f(x)-k\|z-x\|,
\qquad z\in X.
\]
Consequently,
\[
\partial^\sharp f(x)\neq\varnothing
\quad\Longleftrightarrow\quad
f\text{ is }\LCm\text{-subdifferentiable at }x.
\]
The additional information retained by \(\partial^\sharp f(x)\) is metric:
by Proposition~\ref{prop:left_closed_subdiff},
\[
\partial^\sharp f(x)
=
[|\partial^\sharp f|(x),+\infty),
\]
whenever it is nonempty, so that
\(|\partial^\sharp f|(x)\) is the minimal slope of a radial norm-cone
support at \(x\).

The same quantity also admits a precise interpretation in terms of the weak
subdifferential introduced in \cite{DincYalcinKasimbeyli2020}. Recall that
\(
(x^*,\alpha)\in\partial^w f(x)
\quad\Longleftrightarrow\quad
f(z)\ge
f(x)+\langle x^*,z-x\rangle-\alpha\|z-x\|,
\quad z\in X.
\)
Hence,
\(
\alpha\in\partial^\sharp f(x)
\quad\Longleftrightarrow\quad
(0,\alpha)\in\partial^w f(x),
\)
while
\(
(x^*,\alpha)\in\partial^w f(x)
\quad\Longrightarrow\quad
\alpha+\|x^*\|\in\partial^\sharp f(x).
\)
These relations imply that
\(
\partial^\sharp f(x)\neq\varnothing
\quad\Longleftrightarrow\quad
\partial^w f(x)\neq\varnothing,
\)
and
\[
|\partial^\sharp f|(x)
=
\inf
\bigl\{
\alpha+\|x^*\|:
(x^*,\alpha)\in\partial^w f(x)
\bigr\}.
\]
Thus, the minimal norm-cone support slope is exactly the least radial cost
associated with a weak subgradient, once the norm of its linear component has
been absorbed into the radial coefficient.
\end{remark}

The next result is the radial norm-cone counterpart of the elementary optimality
conditions obtained through \(\LCm\)-subdifferentials
\cite{Gorokhovik2022}: \(0\in\partial^\sharp f(\bar x)\) means that the
horizontal support \(z\mapsto f(\bar x)\) lies below \(f\).

\begin{proposition}
\label{prop:minimal_slope_minimizers}
Let \(X\) be a normed space, let
\(f:X\to\overR\) be proper, and let
\(\bar x\in\operatorname{dom}f\). Then the following assertions are equivalent:
\begin{enumerate}
\item \(f\) attains a global minimum at \(\bar x\);

\item \(0\in\partial^\sharp f(\bar x)\);

\item \(|\partial^\sharp f|(\bar x)=0\).
\end{enumerate}
Moreover, if any of these conditions holds, then
\(f\in\mathcal F_{\mathrm{nc}}(X)\). The converse implication does not hold in
general.
\end{proposition}

\begin{proof}
By definition,
\[
0\in\partial^\sharp f(\bar x)
\quad\Longleftrightarrow\quad
f(z)\ge f(\bar x)
\quad\forall z\in X,
\]
which is equivalent to \(\bar x\) being a global minimizer. Moreover, by
Proposition~\ref{prop:left_closed_subdiff}, whenever
\(\partial^\sharp f(\bar x)\neq\varnothing\),
\[
\partial^\sharp f(\bar x)
=
[|\partial^\sharp f|(\bar x),+\infty),
\]
so \(0\in\partial^\sharp f(\bar x)\) if and only if
\(|\partial^\sharp f|(\bar x)=0\).

If these conditions hold, then \(f(z)\ge f(\bar x)\) for every \(z\in X\),
hence \(0\in\slp(f)\) and \(f\in\Fnc(X)\). The converse fails, for example,
for \(g(x)=-\sqrt{|x|}\), which belongs to \(\Fnc(\R)\) since
\[
-\sqrt{|x|}\ge-|x|-1,
\]
but has no global minimizer.
\end{proof}

The density theorem for \(\LCm\)-subdifferentiability due to Gorokhovik
\cite{Gorokhovik2022} transfers directly to the present setting, since
\(\NCm\)- and \(\LCm\)-convexity coincide by
Theorem~\ref{thm:normcone_lcm_equivalence}.

\begin{corollary}
\label{cor:density_norm_cone_supportability}
Let \(X\) be a Banach space and let
\(f:X\to\overR\) be \(\NCm\)-convex. Then the set
\[
\bigl\{
x\in\operatorname{dom} f:
\partial^\sharp f(x)\neq\varnothing
\bigr\}
\]
is dense in \(\operatorname{dom} f\).
\end{corollary}

\begin{proof}
By Theorem~\ref{thm:normcone_lcm_equivalence}, \(f\) is \(\LCm\)-convex. Hence,
by \cite[Theorem~3.12]{Gorokhovik2022}, the set of points where \(f\) is
\(\LCm\)-subdifferentiable is dense in \(\operatorname{dom} f\).

By \cite[Theorem~3.8]{Gorokhovik2022}, \(\LCm\)-subdifferentiability at
\(x\in\operatorname{dom}f\) is equivalent to the existence of \(k\ge0\) such that
\[
f(z)\ge f(x)-k\|z-x\|,
\qquad \forall z\in X.
\]
This is precisely \(k\in\partial^\sharp f(x)\). Therefore
\(\{x\in\operatorname{dom}f:\partial^\sharp f(x)\neq\varnothing\}\) is dense in
\(\operatorname{dom}f\).
\end{proof}

\section{Duality in Norm-cone Optimization}
\label{sec:dual_norm_cone_optimization}
The aim of this section is to develop a duality scheme based on norm-cone conjugation.  From the viewpoint of nonconvex optimization, the main feature of the construction is a substantial reduction of the effective dual
parameterization.

The construction follows the perturbation viewpoint: deviations from the
nominal problem are encoded by a variable \(z\in Z\), while dual information
is extracted through norm-cone conjugation in the perturbation variable.
Although the conjugate itself depends on both a slope and a centre, we shall
see that, in the complete biconjugate expression at the nominal
perturbation, optimization over the centre is attained at \(0_Z\). Hence the
dual problem can be reduced to the slope as its effective parameter without
changing the optimal value.

Let \((X,\|\cdot\|)\) and \((Z,\|\cdot\|)\) be real normed spaces, and let
\(\Omega\subset X\) be nonempty. We consider the primal problem
\begin{equation}\label{eq:primal_problem}
\inf\{f(x):x\in\Omega\},
\qquad (P(0_Z)).
\end{equation}

\subsection{Dual Problem with Respect to Perturbations}
\label{subsec:dual_wrt_perturbations}

Let \(\Phi:X\times Z\to\overline{\mathbb R}\) be an extended-real-valued
perturbation function satisfying
\begin{equation}\label{eq:perturbation_consistency}
\Phi(x,0_Z)=f(x),
\qquad \forall x\in\Omega,
\end{equation}
and
\[
\Phi(x,0_Z)=+\infty,
\qquad \forall x\in X\setminus\Omega.
\]
Thus the constraint \(x\in\Omega\) is incorporated into the perturbation
function.

\begin{definition}
\label{def:value_function_duality_section}
The value function associated with \(\Phi\) is
\begin{equation}\label{eq:value_function_duality_section}
v_\Phi(z):=\inf_{x\in X}\Phi(x,z),
\qquad z\in Z.
\end{equation}
In particular,
\begin{equation}\label{eq:v0_equals_primal_value}
v_\Phi(0_Z)
=
\inf_{x\in X}\Phi(x,0_Z)
=
\inf_{x\in\Omega}f(x).
\end{equation}
\end{definition}

Perturbation duality is based on the partial norm-cone conjugate. The notation
\(\sharp_Z\) means conjugation only in the perturbation variable \(z\).

\begin{definition}
\label{def:global_norm_cone_conjugate_perturbation}
The global norm-cone conjugate of \(\Phi\) with respect to the perturbation
variable is the function
\(\Phi^{\sharp_Z}:\mathbb R_+\times Z\to\overline{\mathbb R}\) defined by
\begin{equation}\label{eq:global_norm_cone_conjugate_perturbation}
\Phi^{\sharp_Z}(\alpha,z)
:=
\sup_{(x',z')\in X\times Z}
\bigl(
-\alpha\|z'-z\|-\Phi(x',z')
\bigr).
\end{equation}
\end{definition}

\begin{proposition}
\label{prop:partial_conjugate_value_function}
Assume that \(v_\Phi\) is proper. Then
\[
\Phi^{\sharp_Z}(\alpha,z)
=
v_\Phi^\sharp(\alpha,z),
\qquad
(\alpha,z)\in\mathbb R_+\times Z.
\]
\end{proposition}

\begin{proof}
For every \((\alpha,z)\in\mathbb R_+\times Z\),
\begin{align*}
\Phi^{\sharp_Z}(\alpha,z)
&=
\sup_{(x',z')\in X\times Z}
\bigl(
-\alpha\|z'-z\|-\Phi(x',z')
\bigr)\\
&=
\sup_{z'\in Z}
\left(
-\alpha\|z'-z\|
-\inf_{x'\in X}\Phi(x',z')
\right)\\
&=
\sup_{z'\in Z}
\bigl(
-\alpha\|z'-z\|-v_\Phi(z')
\bigr)\\
&=
v_\Phi^\sharp(\alpha,z).
\end{align*}
\end{proof}

Although the partial conjugate \(\Phi^{\sharp_Z}\) depends on a centre
\(z\in Z\), Proposition~\ref{prop:dual_value_biconjugate_value_function}
shows that, for every fixed slope \(\alpha\), the complete expression
\[
-\Phi^{\sharp_Z}(\alpha,z)-\alpha\|z\|,
\]
is maximized at the nominal centre \(z=0_Z\). Thus the centre is redundant
for the optimal dual value, even though it remains a genuine parameter of
the partial conjugate itself.

\begin{definition}
\label{def:norm_cone_dual_problem}
The norm-cone dual problem associated with \(P(0_Z)\) is
\begin{equation}\label{eq:norm_cone_dual_problem}
\sup_{\alpha\in\R_+}
\bigl\{
-\Phi^{\sharp_Z}(\alpha,0_Z)
\bigr\},
\qquad
(D^{\sharp_Z}).
\end{equation}
Its optimal value is denoted by \(\sup D^{\sharp_Z}\).
\end{definition}

The following result shows that this reduced formulation is equivalent to
allowing arbitrary centres in the dual construction.

\begin{proposition}
\label{prop:dual_value_biconjugate_value_function}
Assume that \(v_\Phi\) is proper. Then, for every \(\alpha\ge0\),
\[
\sup_{z\in Z}
\bigl\{
-\Phi^{\sharp_Z}(\alpha,z)-\alpha\|z\|
\bigr\}
=
-\Phi^{\sharp_Z}(\alpha,0_Z).
\]
Consequently,
\[
\sup D^{\sharp_Z}
=
\sup_{(\alpha,z)\in\R_+\times Z}
\bigl\{
-\Phi^{\sharp_Z}(\alpha,z)-\alpha\|z\|
\bigr\}
=
v_\Phi^{\sharp\sharp}(0_Z).
\]
\end{proposition}

\begin{proof}
By Proposition~\ref{prop:partial_conjugate_value_function},
\[
\Phi^{\sharp_Z}(\alpha,z)
=
v_\Phi^\sharp(\alpha,z).
\]
For every \(u,z\in Z\), the triangle inequality gives
\[
-\alpha\|u-z\|
\ge
-\alpha\|u\|-\alpha\|z\|.
\]
Therefore,
\[
v_\Phi^\sharp(\alpha,z)
=
\sup_{u\in Z}
\bigl(
-\alpha\|u-z\|-v_\Phi(u)
\bigr)
\ge
v_\Phi^\sharp(\alpha,0_Z)-\alpha\|z\|.
\]
Hence
\[
-v_\Phi^\sharp(\alpha,z)-\alpha\|z\|
\le
-v_\Phi^\sharp(\alpha,0_Z).
\]
Equality is attained for \(z=0_Z\), and thus
\[
\sup_{z\in Z}
\bigl\{
-v_\Phi^\sharp(\alpha,z)-\alpha\|z\|
\bigr\}
=
-v_\Phi^\sharp(\alpha,0_Z).
\]
Using Proposition~\ref{prop:partial_conjugate_value_function} again gives the
first assertion. Taking the supremum over \(\alpha\ge0\), we obtain
\[
\sup D^{\sharp_Z}
=
\sup_{(\alpha,z)\in\R_+\times Z}
\bigl\{
-v_\Phi^\sharp(\alpha,z)-\alpha\|z\|
\bigr\}
=
v_\Phi^{\sharp\sharp}(0_Z),
\]
by the definition of the norm-cone biconjugate.
\end{proof}

A distinctive feature of the resulting duality scheme is its dimensional
reduction: although both the decision and perturbation spaces may be
vector-valued, or even infinite-dimensional, the effective norm-cone dual
problem is parametrized by a single real variable, namely the slope
\(\alpha\ge 0\).

\begin{example}
The preceding reduction turns the following nonconvex problem with two
inequality constraints into a scalar dual program. Define
\(f(x_1,x_2):=(x_1+x_2)^2+(x_1x_2-1)^2\) and
\(g(x_1,x_2):=(1-x_1x_2,x_1-x_2)\), and consider
\begin{equation}\label{eq:primal-example}
(P)\qquad
\inf\bigl\{f(x):g(x)\leq_{\mathbb R_+^2}0_{\mathbb R^2}\bigr\}.
\end{equation}
The objective is nonconvex, and so is the feasible set, since it contains
\((1,1)\) and \((-1,-1)\), but not their midpoint. The constraints are
\(x_1x_2\geq1\) and \(x_1\leq x_2\). Hence, on the feasible set,
\(f(x)\geq4x_1x_2+(x_1x_2-1)^2=(x_1x_2+1)^2\geq4\), with equality at
\((1,1)\) and \((-1,-1)\). Therefore \(\inf(P)=4\).

Let \(Z=\mathbb R^2\), endowed with the Euclidean norm and ordered by
\(\mathbb R_+^2\), and perturb both inequalities independently by setting
\[
\Phi(x,z):=
\begin{cases}
f(x), & g(x)\leq_{\mathbb R_+^2}z,\\
+\infty, & \text{otherwise}.
\end{cases}
\]
Thus the perturbation variable \(z=(z_1,z_2)\) is two-dimensional. By
Proposition~\ref{prop:dual_value_biconjugate_value_function}, however, the
norm-cone dual has the scalar form
\begin{equation}\label{eq:dual-example}
(D^{\sharp_Z})\qquad
\sup_{\alpha\geq0}\bigl\{-\Phi^{\sharp_Z}(\alpha,0_Z)\bigr\}.
\end{equation}
For fixed \(x\), the minimum of \(\|z\|_2\) subject to \(g(x)\leq z\) is
\(\|g(x)_+\|_2\), where the positive part is taken componentwise.
Consequently,
\begin{equation}\label{eq:penalty-representation-example}
-\Phi^{\sharp_Z}(\alpha,0_Z)
=
\inf_{x\in\mathbb R^2}
\left\{
f(x)+\alpha
\sqrt{(1-x_1x_2)_+^2+(x_1-x_2)_+^2}
\right\}.
\end{equation}

Since \(f\) and \(x_1x_2\) are invariant under permutation of the coordinates,
the infimum may be restricted to \(x_1\leq x_2\); on this set,
\((x_1-x_2)_+=0\). Setting \(t=x_1x_2-1\), standard minimization gives
\[
q(t):=
\inf_{\substack{x_1\leq x_2\\x_1x_2=1+t}}f(x_1,x_2)
=
\begin{cases}
t^2, & t\leq-1,\\
(t+2)^2, & t\geq-1.
\end{cases}
\]
It follows that
\(-\Phi^{\sharp_Z}(\alpha,0_Z)
=\inf_{t\in\mathbb R}\{q(t)+\alpha(-t)_+\}\).
Elementary one-variable minimization yields
\begin{equation}\label{eq:dual-objective-example}
-\Phi^{\sharp_Z}(\alpha,0_Z)
=
\begin{cases}
1+\alpha, & 0\leq\alpha\leq2,\\[1mm]
\displaystyle
4-\frac{(4-\alpha)^2}{4}, & 2\leq\alpha\leq4,\\[2mm]
4, & \alpha\geq4.
\end{cases}
\end{equation}
Therefore \(\sup(D^{\sharp_Z})=4=\inf(P)\) and
\(\operatorname{Sol}(D^{\sharp_Z})=[4,+\infty)\). Thus \(\alpha^*=4\) is
both the smallest optimal dual slope and, by
\eqref{eq:penalty-representation-example}, the least exact distance-penalty
parameter. Hence a nonconvex problem with two inequality constraints and a
genuinely two-dimensional perturbation gives rise to a norm-cone dual whose
only effective variable is the scalar slope \(\alpha\geq0\).
\end{example}

\begin{theorem}
\label{thm:weak_norm_cone_duality}
One has
\[
\sup D^{\sharp_Z}
\leq
\inf_{x\in\Omega}f(x).
\]
\end{theorem}

\begin{proof}
Let \(\alpha\ge0\). By the definition of the partial norm-cone conjugate,
for every \(x\in X\),
\[
\Phi^{\sharp_Z}(\alpha,0_Z)
\ge
-\Phi(x,0_Z).
\]
Hence
\[
-\Phi^{\sharp_Z}(\alpha,0_Z)
\le
\inf_{x\in X}\Phi(x,0_Z)
=
v_\Phi(0_Z)
=
\inf_{x\in\Omega}f(x).
\]
Taking the supremum over \(\alpha\ge0\) yields
\[
\sup D^{\sharp_Z}
\le
\inf_{x\in\Omega}f(x).
\]
\end{proof}

\begin{corollary}
\label{cor:strong_duality_pointwise_biconjugation}
Assume that \(v_\Phi\) is proper and \(v_\Phi(0_Z)\in\mathbb R\).
Then
\[
\sup D^{\sharp_Z}
=
\inf_{x\in\Omega}f(x)
\quad\Longleftrightarrow\quad
v_\Phi^{\sharp\sharp}(0_Z)
=
v_\Phi(0_Z).
\]
Thus strong norm-cone duality is equivalent to exact norm-cone
biconjugation of the value function at the nominal perturbation \(0_Z\).
\end{corollary}

\begin{proof}
By Proposition~\ref{prop:dual_value_biconjugate_value_function},
\[
\sup D^{\sharp_Z}
=
v_\Phi^{\sharp\sharp}(0_Z),
\]
whereas
\[
v_\Phi(0_Z)
=
\inf_{x\in\Omega}f(x)
\]
by~\eqref{eq:v0_equals_primal_value}.
\end{proof}

\begin{remark}
\label{rem:weak_conjugacy_dual_value}
Whenever \(v_\Phi\in\Fnc(Z)\), Remark~\ref{rem:weak_conjugacy_biconjugate}
gives
\[
v_\Phi^{ww}(0_Z)
=
v_\Phi^{\sharp\sharp}(0_Z)
=
\sup D^{\sharp_Z}.
\]
Thus the weak-conjugacy dual of \cite{DincYalcinKasimbeyli2020} and the
present norm-cone dual have the same optimal value on this class, although
their parametrizations and sets of dual solutions need not coincide. In the
present construction the effective dual parameter reduces to the radial
slope.

\end{remark}

\subsection{Strong Norm-cone Duality from Lower Lipschitz Bounds}
\label{subsec:strong_norm_cone_duality_uniform_lower_lipschitz}
Whenever \(v_\Phi\) is proper,
Corollary~\ref{cor:strong_duality_pointwise_biconjugation}
shows that strong norm-cone duality is equivalent to exact biconjugation
of the value function at \(0_Z\). A sufficient condition for this
pointwise exactness is the existence of a norm-cone support of
\(v_\Phi\) at \(0_Z\). We next obtain such a support from a uniform
lower Lipschitz estimate on the perturbation function.

\begin{definition}
\label{def:ULL}
We say that \(\Phi\) is uniformly lower Lipschitz at \(0_Z\) in the perturbation
variable if there exists \(\alpha\ge0\) such that
\begin{equation}\label{eq:ULL}
\Phi(x,z)\ge \Phi(x,0_Z)-\alpha\|z\|,
\qquad \forall (x,z)\in X\times Z.
\end{equation}
\end{definition}

\begin{theorem}
\label{thm:strong_norm_cone_duality_ULL}
Assume that \(v_\Phi(0_Z)\in\R\). If \(\Phi\) is uniformly lower Lipschitz
at \(0_Z\) in the perturbation variable, then
\begin{equation}\label{eq:strong_norm_cone_duality}
\inf_{x\in\Omega}f(x)
=
\sup D^{\sharp_Z}
=
\sup_{\beta\in\R_+}
\bigl\{
-\Phi^{\sharp_Z}(\beta,0_Z)
\bigr\}.
\end{equation}
\end{theorem}
\begin{proof}
Let \(\alpha\geq0\) satisfy~\eqref{eq:ULL}. Taking the infimum over
\(x\in X\) gives

\[
v_\Phi(z)
\geq
v_\Phi(0_Z)-\alpha\|z\|,
\qquad z\in Z.
\]

Since \(v_\Phi(0_Z)\in\mathbb R\), this estimate implies that
\(v_\Phi\) is proper, belongs to \(\mathcal F_{\rm nc}(Z)\), and
\[
\alpha\in\partial^\sharp v_\Phi(0_Z).
\]

Hence \(v_\Phi\) admits at \(0_Z\) the norm-cone support

\[
z\longmapsto v_\Phi(0_Z)-\alpha\|z\|.
\]

Now, by Corollary~\ref{cor:pointwise_exactness_norm_cone_support}, we have

\[
v_\Phi^{\sharp\sharp}(0_Z)=v_\Phi(0_Z).
\]

By Proposition~\ref{prop:dual_value_biconjugate_value_function},

\[
\sup D^{\sharp_Z}
=
v_\Phi^{\sharp\sharp}(0_Z)
=
v_\Phi(0_Z)
=
\inf_{x\in\Omega}f(x),
\]

which proves the result.
\end{proof}

\begin{remark}
\label{rem:rockafellar_norm_cone_duality}
In classical perturbation duality, dual bounds arise from linear supports
of the value function. Here they arise from radial supports
\[
z\longmapsto v_\Phi(0_Z)-\alpha\|z\|,
\]
with the slope \(\alpha\) as the effective dual parameter.
\end{remark}

\section{Distance Penalization and Norm-cone Duality}
\label{sec:distance_penalization_norm_cone_duality}
We now apply the abstract norm-cone duality framework developed in the
preceding section to conic inequality problems.  The optimization significance of the scalar dual parameter becomes especially
transparent in this setting, where the same slope admits a direct
interpretation as an exact distance-penalty parameter. The main objective is to
connect the natural value function associated with the constraint
perturbation to exact distance penalization, norm-cone support at the
nominal perturbation, and global error-bound conditions. This will provide
both a characterization of exact penalty parameters and verifiable
sufficient conditions for strong norm-cone duality.

Let \((X,\|\cdot\|)\) and \((Z,\|\cdot\|)\) be real normed spaces, let
\(\Omega\subset X\) be nonempty, and let \(Z_+\subset Z\) be a nonempty closed
cone. Let \(f:\Omega\to\R\) and \(g:\Omega\to Z\). We consider the conic
inequality problem
\begin{equation}\label{eq:P0_conic}
\inf\{f(x):x\in\Omega,\ g(x)\in -Z_+\},
\qquad (P(0_Z)).
\end{equation}
Its feasible set is
\[
\mathcal F:=\{x\in\Omega:\ g(x)\in -Z_+\}.
\]

The natural perturbation associated with
\eqref{eq:P0_conic} is obtained by shifting the target cone. Define
\[
\Phi:X\times Z\to\overR
\]
by
\begin{equation}\label{eq:conic_perturbation_function}
\Phi(x,z)
:=
\begin{cases}
f(x), & x\in\Omega,\ g(x)\in z-Z_+,\\
+\infty, & \text{otherwise}.
\end{cases}
\end{equation}
Its associated value function is
\begin{equation}\label{eq:value_function_conic_Phi}
v_\Phi(z)
:=
\inf_{x\in X}\Phi(x,z)
=
\inf\{f(x):x\in\Omega,\ g(x)\in z-Z_+\},
\qquad z\in Z.
\end{equation}
In particular,
\[
v_\Phi(0_Z)
=
\inf_{x\in\mathcal F}f(x).
\]

The abstract duality theory developed in
Section~\ref{sec:dual_norm_cone_optimization} can be applied directly to
\(\Phi\). The corresponding partial norm-cone conjugate is
\[
\Phi^{\sharp_Z}:\R_+\times Z\to\overR,
\]
defined by
\begin{equation}\label{eq:Phi_sharp_Z_conic}
\Phi^{\sharp_Z}(\alpha,z)
:=
\sup_{(x',z')\in X\times Z}
\bigl(
-\alpha\|z'-z\|-\Phi(x',z')
\bigr).
\end{equation}
The associated norm-cone dual problem, in the reduced form established in
Proposition~\ref{prop:dual_value_biconjugate_value_function}, is
\begin{equation}\label{eq:D_direct_conic}
\sup_{\alpha\in\R_+}
\bigl\{
-\Phi^{\sharp_Z}(\alpha,0_Z)
\bigr\},
\qquad
(D^{\sharp_Z}).
\end{equation}

However, the perturbation \(\Phi\) is typically singular in the perturbation
variable \(z\), since it takes the value \(+\infty\) outside the shifted
constraint system. Consequently, the uniform lower Lipschitz condition
required by the abstract strong duality theorem may fail even in simple
situations. This motivates the use of distance penalization, which we now
relate to the original constraint perturbation.

\begin{definition}
\label{def:exact_distance_penalty}
A number \(\lambda\geq0\) is called an \emph{exact distance penalty
parameter} for problem~\eqref{eq:P0_conic} if
\begin{equation}
\label{eq:exact_penalty_conic}
\inf_{x\in\Omega}
\bigl(f(x)+\lambda d_{-Z_+}(g(x))\bigr)=\inf_{x\in\mathcal F}f(x).
\end{equation}
We say that problem~\eqref{eq:P0_conic} admits an exact distance
penalization if it has at least one exact distance penalty parameter.
\end{definition}

The perturbation variable \(z\) shifts the target cone. Penalizing this
perturbation by \(\lambda\|z\|\) is equivalent to penalizing the original
constraint violation by the distance to \(-Z_+\).

\begin{lemma}
\label{lem:distance_representation_penalized_value}
For every \(\lambda\geq0\),
\begin{equation}
\label{eq:penalty_as_norm_cone_dual_value}
\inf_{z\in Z}
\bigl(v_\Phi(z)+\lambda\|z\|\bigr)=\inf_{x\in\Omega}
\bigl(f(x)+\lambda d_{-Z_+}(g(x))\bigr).
\end{equation}
If \(v_\Phi(0_Z)\in\mathbb R\), then \(\lambda\) is an exact distance
penalty parameter if and only if
\begin{equation}
\label{eq:exact_penalty_value_support}
v_\Phi(z)
\geq
v_\Phi(0_Z)-\lambda\|z\|,
\qquad z\in Z.
\end{equation}
\end{lemma}

\begin{proof}
Using~\eqref{eq:value_function_conic_Phi},
\begin{align*}
\inf_{z\in Z}\bigl(v_\Phi(z)+\lambda\|z\|\bigr)
&=
\inf_{z\in Z}
\inf_{\substack{x\in\Omega\\ g(x)\in z-Z_+}}
\bigl(f(x)+\lambda\|z\|\bigr)\\
&=
\inf_{x\in\Omega}
\left(
f(x)+
\lambda
\inf_{\substack{z\in Z\\ g(x)\in z-Z_+}}
\|z\|
\right).
\end{align*}
For fixed \(x\in\Omega\), the condition
\(g(x)\in z-Z_+\) is equivalent to \(z\in g(x)+Z_+\). Hence

\[
\inf_{\substack{z\in Z\\ g(x)\in z-Z_+}}\|z\|
=
\inf_{z\in g(x)+Z_+}\|z\|
=
\inf_{w\in -Z_+}\|g(x)-w\|
=
d_{-Z_+}(g(x)),
\]

which proves~\eqref{eq:penalty_as_norm_cone_dual_value}.

Now assume that \(v_\Phi(0_Z)\in\mathbb R\). Since

\[
v_\Phi(0_Z)=\inf_{x\in\mathcal F}f(x),
\]

Definition~\ref{def:exact_distance_penalty} and
\eqref{eq:penalty_as_norm_cone_dual_value} show that \(\lambda\) is an
exact distance penalty parameter if and only if

\[
\inf_{z\in Z}
\bigl(v_\Phi(z)+\lambda\|z\|\bigr)
=
v_\Phi(0_Z).
\]

Since \(z=0_Z\) is admissible in the infimum, the left-hand side never
exceeds \(v_\Phi(0_Z)\). Therefore equality holds if and only if

\[
v_\Phi(z)+\lambda\|z\|
\geq
v_\Phi(0_Z),
\qquad z\in Z,
\]

which is precisely~\eqref{eq:exact_penalty_value_support}.
\end{proof}

Notice also that exactness is monotone in the penalty parameter: if
\(\bar\lambda\) is exact, then every \(\lambda\ge\bar\lambda\) is exact,
since
\[
f(x)+\lambda d_{-Z_+}(g(x))
\ge
f(x)+\bar\lambda d_{-Z_+}(g(x)).
\]

\begin{example}
\label{ex:direct_vs_penalized_perturbation}
Let

\[
X=Z=\R,\qquad Z_+=\R_+,\qquad \Omega=[-1,1],
\]

and define

\[
f(x):=-x,
\qquad
g(x):=x.
\]

The constraint \(g(x)\in -Z_+\) is simply \(x\le0\). Hence

\[
\mathcal F=[-1,0],
\qquad
\inf_{x\in\mathcal F}f(x)=0.
\]

\emph{(i) Failure of the direct perturbation.}
The natural constraint perturbation is

\[
\Phi(x,z)
:=
\begin{cases}
-x, & x\in[-1,1],\ x\le z,\\
+\infty, & \text{otherwise}.
\end{cases}
\]

It does not satisfy the uniform lower Lipschitz condition. Indeed,

\[
\Phi(1/2,1/2)=-1/2,
\qquad
\Phi(1/2,0)=+\infty,
\]

so there is no \(\alpha\ge0\) such that

\[
\Phi(x,z)\ge \Phi(x,0)-\alpha\|z\|,
\qquad (x,z)\in X\times Z.
\]

\emph{(ii) Distance penalization and regularization.}
Since

\[
d_{-\R_+}(u)=u_+,
\]

the distance-penalized objective at the nominal perturbation is

\[
f(x)+\lambda d_{-\R_+}(g(x))
=
-x+\lambda x_+.
\]

For \(x\le0\), this quantity equals \(-x\ge0\), whereas for \(x>0\)
it equals \((\lambda-1)x\). Consequently,

\[
\inf_{x\in[-1,1]}
\bigl(-x+\lambda x_+\bigr)
=
0
=
\inf_{x\in\mathcal F}f(x),
\qquad \lambda\ge1.
\]

On the other hand, if \(0\le\lambda<1\), evaluating at \(x=1\) gives

\[
-1+\lambda<0
=
\inf_{x\in\mathcal F}f(x).
\]

Thus the distance penalization is exact if and only if
\(\lambda\ge1\), and its least exact penalty parameter is

\[
\lambda_*=1.
\]

This threshold also has a direct interpretation in terms of the
norm-cone geometry of the natural value function. Indeed,

\[
v_\Phi(z)
=
\inf\{-x:x\in[-1,1],\ x\leq z\}
=
\begin{cases}
+\infty, & z<-1,\\
-z, & -1\leq z\leq1,\\
-1, & z\geq1.
\end{cases}
\]

Since \(v_\Phi(0)=0\), the condition

\[
v_\Phi(z)\geq v_\Phi(0)-\lambda|z|,
\qquad z\in\mathbb R,
\]

holds if and only if \(\lambda\geq1\). Therefore

\[
\partial^\sharp v_\Phi(0)=[1,+\infty)
\qquad\text{and}\qquad
|\partial^\sharp v_\Phi|(0)=1=\lambda_*.
\]

Thus, in this example, the least exact distance penalty parameter
coincides with the minimal norm-cone support slope of the natural value
function at the nominal perturbation.

For \(\lambda\ge0\), consider now the corresponding penalized perturbation

\[
\Phi_\lambda(x,z)
:=
\begin{cases}
-x+\lambda(x-z)_+, & x\in[-1,1],\\
+\infty, & \text{otherwise}.
\end{cases}
\]

Since \(u\mapsto u_+\) is \(1\)-Lipschitz,

\[
(x-z')_+
\ge
(x-z)_+-|z'-z|,
\]

and hence

\[
\Phi_\lambda(x,z')
\ge
\Phi_\lambda(x,z)-\lambda|z'-z|,
\qquad x\in[-1,1],\quad z,z'\in\R.
\]

Thus \(\Phi_\lambda\) satisfies the uniform lower Lipschitz condition
with constant \(\lambda\). Hence distance penalization provides a
regularized perturbation to which
Theorem~\ref{thm:strong_norm_cone_duality_ULL} applies directly.
Notice, however, that failure of the uniform lower Lipschitz condition
for the original perturbation \(\Phi\) does not by itself imply failure
of strong duality for the direct norm-cone dual problem; as shown below,
the decisive property can instead be expressed in terms of norm-cone
supportability of the associated value function.
\end{example}

The preceding example exhibits the two roles of distance penalization that
will be used below: preservation of the primal value under exactness and
regularization of the perturbation variable. Motivated by this construction,
for \(\lambda\ge0\) define the distance-penalized perturbation

\[
\Phi_\lambda:X\times Z\to\overR
\]

by
\begin{equation}\label{eq:Perturbation_Phi_lambda}
\Phi_\lambda(x,z)
:=
\begin{cases}
f(x)+\lambda d_{-Z_+}(g(x)-z), & x\in\Omega,\\
+\infty, & x\notin\Omega.
\end{cases}
\end{equation}

Its associated value function is
\begin{equation}\label{eq:value_function_conic_Phi_lambda}
v_{\Phi_\lambda}(z)
:=
\inf_{x\in X}\Phi_\lambda(x,z),
\qquad z\in Z.
\end{equation}
In particular,
\[
v_{\Phi_\lambda}(0_Z)
=
\inf_{x\in\Omega}
\bigl(f(x)+\lambda d_{-Z_+}(g(x))\bigr).
\]

The corresponding partial norm-cone conjugate is
\[
\Phi_\lambda^{\sharp_Z}:\R_+\times Z\to\overR,
\]
defined by
\begin{equation}\label{eq:Phi_lambda_sharp_Z_conic}
\Phi_\lambda^{\sharp_Z}(\alpha,z)
:=
\sup_{(x',z')\in X\times Z}
\bigl(
-\alpha\|z'-z\|-\Phi_\lambda(x',z')
\bigr).
\end{equation}
The associated norm-cone dual problem is
\begin{equation}\label{eq:D_lambda_conic}
\sup_{\alpha\in\R_+}
\bigl\{
-\Phi_\lambda^{\sharp_Z}(\alpha,0_Z)
\bigr\},
\qquad
(D_\lambda^{\sharp_Z}).
\end{equation}
Its optimal value is denoted by \(\sup D_\lambda^{\sharp_Z}\).

\subsection{Weak Norm-cone Duality}
\label{subsec:weak_norm_cone_duality_distance}

Weak duality is inherited from the abstract perturbation theorem applied to
\(\Phi_\lambda\).

\begin{theorem}
\label{thm:weak_norm_cone_duality_distance}
For every \(\lambda\ge0\),
\[
\sup D_\lambda^{\sharp_Z}
\le
v_{\Phi_\lambda}(0_Z)
=
\inf_{x\in\Omega}
\bigl(f(x)+\lambda d_{-Z_+}(g(x))\bigr)
\le
\inf_{x\in\mathcal F}f(x).
\]
\end{theorem}

\begin{proof}
The first inequality follows from
Theorem~\ref{thm:weak_norm_cone_duality}
applied to \(\Phi_\lambda\), together with
\eqref{eq:value_function_conic_Phi_lambda}. The equality is exactly
\eqref{eq:value_function_conic_Phi_lambda}. The last inequality follows by
evaluating the penalized objective on \(\mathcal F\), where
\(d_{-Z_+}(g(x))=0\).
\end{proof}

\subsection{Strong Norm-cone Duality from Exact Distance Penalization}
\label{subsec:strong_norm_cone_duality_exact_distance_penalization}

The distance function \(d_{-Z_+}\) is \(1\)-Lipschitz. Consequently, every
penalized perturbation \(\Phi_\lambda\) satisfies the uniform lower Lipschitz
bound required in Theorem~\ref{thm:strong_norm_cone_duality_ULL}. Therefore,
strong norm-cone duality for the constrained problem follows once the
distance penalty is exact.

\begin{theorem}
\label{thm:strong_duality_conic}
Assume that \(\mathcal F\neq\varnothing\), \(\inf_{x\in\mathcal F}f(x)\in\mathbb R\),
and that \eqref{eq:P0_conic} admits an exact distance penalization with
parameter \(\bar\lambda\ge0\). Then, for every \(\lambda\ge\bar\lambda\),
\[
\sup D_\lambda^{\sharp_Z}
=
\inf_{x\in\mathcal F} f(x).
\]
\end{theorem}

\begin{proof}
Since \(d_{-Z_+}\) is \(1\)-Lipschitz, for every \(x\in\Omega\) and every
\(z\in Z\),
\[
d_{-Z_+}(g(x)-z)
\ge
d_{-Z_+}(g(x))-\|z\|.
\]
Thus
\[
f(x)+\lambda d_{-Z_+}(g(x)-z)
\ge
f(x)+\lambda d_{-Z_+}(g(x))-\lambda \|z\|,
\]
or equivalently
\[
\Phi_\lambda(x,z)
\ge
\Phi_\lambda(x,0_Z)-\lambda\|z\|,
\qquad \forall (x,z)\in X\times Z,
\]
where the assertion is trivial for \(x\notin\Omega\). Hence \(\Phi_\lambda\)
satisfies \eqref{eq:ULL} with constant \(\lambda\).

By Theorem~\ref{thm:strong_norm_cone_duality_ULL},
\[
\sup D_\lambda^{\sharp_Z}
=
v_{\Phi_\lambda}(0_Z)
=
\inf_{x\in\Omega}
\bigl(f(x)+\lambda d_{-Z_+}(g(x))\bigr).
\]
Since exactness is monotone with respect to the penalty parameter,
\(\lambda\ge\bar\lambda\) gives
\[
\inf_{x\in\Omega}
\bigl(f(x)+\lambda d_{-Z_+}(g(x))\bigr)
=
\inf_{x\in\mathcal F} f(x).
\]
The result follows.
\end{proof}

\subsection{Exact Distance Penalization: Characterization and Sufficient Conditions}
\label{subsec:sufficient_conditions_exact_distance_penalization}
It remains to identify conditions ensuring exact distance penalization.
We first show that exactness admits a precise characterization in terms
of norm-cone supports of the value function associated with the direct
constraint perturbation. In particular, the exact penalty parameters
are identified with the norm-cone subdifferential of the value function
at the nominal perturbation. We then derive a verifiable sufficient
condition from global error bounds.

\subsubsection{Characterization of Exact Distance Penalization via Norm-cone Support}
\label{subsubsec:exact_penalization_value_function_support}
Lemma~\ref{lem:distance_representation_penalized_value} relates the constraint
perturbation to distance penalization. It shows that exactness is equivalent
to the existence of a norm-cone support of the value function at the nominal
perturbation.

\begin{theorem}[Characterization of exact distance penalization]
\label{thm:exact_penalty_value_function_characterization}
Assume that
\[
v_\Phi(0_Z)=\inf_{x\in\mathcal F}f(x)\in\R,
\]
and let \(\lambda\ge0\). The following assertions are equivalent:
\begin{enumerate}
\item \(\lambda\) is an exact distance penalty parameter, that is,
\[
\inf_{x\in\Omega}
\bigl(
f(x)+\lambda d_{-Z_+}(g(x))
\bigr)
=
\inf_{x\in\mathcal F}f(x);
\]

\item
\[
v_\Phi(z)\ge v_\Phi(0_Z)-\lambda\|z\|,
\qquad z\in Z;
\]

\item \(v_\Phi\) is proper and
\[
\lambda\in\partial^\sharp v_\Phi(0_Z);
\]

\item \(v_\Phi\) is proper and
\[
v_\Phi^\sharp(\lambda,0_Z)=-v_\Phi(0_Z).
\]
\end{enumerate}

Consequently,
\[
\bigl\{
\lambda\ge0:
\lambda \text{ is an exact distance penalty parameter}
\bigr\}
=
\partial^\sharp v_\Phi(0_Z).
\]
In particular, problem~\eqref{eq:P0_conic} admits an exact distance
penalization if and only if
\[
\partial^\sharp v_\Phi(0_Z)\neq\varnothing.
\]
Whenever these equivalent conditions hold,
\[
\partial^\sharp v_\Phi(0_Z)
=
[\lambda^*,+\infty),
\qquad
\lambda^*
=
|\partial^\sharp v_\Phi|(0_Z),
\]
and \(\lambda^*\) is the least exact distance penalty parameter.
\end{theorem}

\begin{proof}
The equivalence between \emph{(1)} and \emph{(2)} follows directly
from Lemma~\ref{lem:distance_representation_penalized_value}.

Condition \emph{(2)} implies that \(v_\Phi\) never takes the value
\(-\infty\), and since \(v_\Phi(0_Z)\in\mathbb R\), the function
\(v_\Phi\) is proper. Moreover,

\[
v_\Phi(z)
\geq
v_\Phi(0_Z)-\lambda\|z-0_Z\|,
\qquad z\in Z,
\]

which is precisely

\[
\lambda\in\partial^\sharp v_\Phi(0_Z).
\]

Thus \emph{(2)} and \emph{(3)} are equivalent.

Finally,

\[
v_\Phi^\sharp(\lambda,0_Z)
=
\sup_{z\in Z}
\bigl(-\lambda\|z\|-v_\Phi(z)\bigr).
\]

Condition \emph{(2)} is equivalent to

\[
-\lambda\|z\|-v_\Phi(z)
\leq
-v_\Phi(0_Z),
\qquad z\in Z.
\]

Since equality is attained at \(z=0_Z\), this is equivalent to

\[
v_\Phi^\sharp(\lambda,0_Z)
=
-v_\Phi(0_Z),
\]

and hence \emph{(2)} and \emph{(4)} are equivalent.

Since the equivalence between \emph{(1)} and \emph{(3)} holds for every
\(\lambda\ge0\), we obtain
\[
\bigl\{
\lambda\ge0:
\lambda \text{ is an exact distance penalty parameter}
\bigr\}
=
\partial^\sharp v_\Phi(0_Z).
\]
Therefore, problem~\eqref{eq:P0_conic} admits an exact distance penalization
if and only if
\[
\partial^\sharp v_\Phi(0_Z)\neq\varnothing.
\]
If this set is nonempty, Proposition~\ref{prop:left_closed_subdiff} gives
\[
\partial^\sharp v_\Phi(0_Z)
=
[|\partial^\sharp v_\Phi|(0_Z),+\infty).
\]
Hence
\[
\lambda^*
:=
|\partial^\sharp v_\Phi|(0_Z)
\]
is the least exact distance penalty parameter.
\end{proof}

\begin{corollary}
\label{cor:exact_penalty_pointwise_biconjugation}
Assume that \(v_\Phi(0_Z)\in\mathbb R\). If the problem admits an exact
distance penalization, then

\[
v_\Phi^{\sharp\sharp}(0_Z)=v_\Phi(0_Z).
\]

Consequently, the direct norm-cone dual problem satisfies

\[
\sup D^{\sharp_Z}
=
\inf_{x\in\mathcal F}f(x).
\]

\end{corollary}

\begin{proof}
Let \(\lambda\) be an exact distance penalty parameter. By
Theorem~\ref{thm:exact_penalty_value_function_characterization},
\(v_\Phi\) is proper and
\(\lambda\in\partial^\sharp v_\Phi(0_Z)\). In particular,
\(v_\Phi\in\mathcal F_{\rm nc}(Z)\) and \(v_\Phi\) admits a
norm-cone support at \(0_Z\). Hence Corollary~\ref{cor:pointwise_exactness_norm_cone_support} gives

\[
v_\Phi^{\sharp\sharp}(0_Z)=v_\Phi(0_Z).
\]

Using Proposition~\ref{prop:dual_value_biconjugate_value_function},

\[
\sup D^{\sharp_Z}
=
v_\Phi^{\sharp\sharp}(0_Z)
=
v_\Phi(0_Z)
=
\inf_{x\in\mathcal F}f(x).
\]

\end{proof}
\begin{remark}
\label{rem:direct_vs_penalized_duality}
The direct and penalized perturbation schemes provide two different
routes to strong norm-cone duality. A uniform lower Lipschitz estimate
for the original perturbation \(\Phi\) is a sufficient condition for
strong duality through Theorem~\ref{thm:strong_norm_cone_duality_ULL}.
However, it is not necessary. Indeed, as shown in
Theorem~\ref{thm:exact_penalty_value_function_characterization}, exact
distance penalization is equivalent to the existence of a norm-cone
support of the direct value function \(v_\Phi\) at \(0_Z\), and hence
implies

\[
v_\Phi^{\sharp\sharp}(0_Z)=v_\Phi(0_Z).
\]

Therefore, if \(\lambda\) is an exact distance penalty parameter,
the direct dual problem and the penalized dual problem
\(D_\lambda^{\sharp_Z}\) both have the primal optimal value, even if
the original perturbation \(\Phi\) fails the uniform lower Lipschitz
condition.
\end{remark}

\subsubsection{Exact Penalization from Global Error Bounds}
\label{subsubsec:exact_penalization_global_error_bounds}

A second mechanism is a global error bound for the constraint system. It
estimates the distance to the feasible set by the constraint residual.

Define
\[
G(x):=g(x)+Z_+,
\qquad x\in\Omega.
\]
Then \(G^{-1}(0_Z)=\mathcal F\), and, for every \(x\in\Omega\),
\[
\operatorname{dist}(0_Z,G(x))
=
\inf_{v\in Z_+}\|g(x)+v\|
=
d_{-Z_+}(g(x)).
\]
\begin{definition}
\label{def:global_error_bound}
We say that the constraint system admits a global error bound on \(\Omega\) if
there exists \(\kappa\ge0\) such that
\[
\operatorname{dist}(x,\mathcal F)
\le
\kappa\,d_{-Z_+}(g(x)),
\qquad \forall x\in\Omega.
\]
\end{definition}

\begin{theorem}
\label{thm:exact_penalty_error_bound}
Assume that \(\mathcal F\neq\varnothing\), \(\inf_{x\in\mathcal F}f(x)\in\mathbb R\),
and that \(f\) is Lipschitz on \(\Omega\). If the global error bound holds with
constant \(\kappa\), then \eqref{eq:exact_penalty_conic} holds for every
\[
\lambda\ge \operatorname{Lip}(f)\kappa.
\]
\end{theorem}

\begin{proof}
Set \(p:=\inf_{x\in\mathcal F}f(x)\). Since \(d_{-Z_+}(g(x))=0\) on
\(\mathcal F\),
\[
\inf_{x\in\Omega}
\bigl(f(x)+\lambda d_{-Z_+}(g(x))\bigr)
\le p.
\]

Conversely, fix \(x\in\Omega\) and \(\varepsilon>0\). Choose
\(x_\varepsilon\in\mathcal F\) such that
\[
\|x-x_\varepsilon\|
\le
\operatorname{dist}(x,\mathcal F)+\varepsilon.
\]
By Lipschitz continuity of \(f\),
\[
p
\le
f(x_\varepsilon)
\le
f(x)+\operatorname{Lip}(f)\|x-x_\varepsilon\|
\le
f(x)+\operatorname{Lip}(f)\operatorname{dist}(x,\mathcal F)
+\operatorname{Lip}(f)\varepsilon.
\]
Using the global error bound,
\[
p
\le
f(x)+\operatorname{Lip}(f)\kappa\,d_{-Z_+}(g(x))
+\operatorname{Lip}(f)\varepsilon.
\]
If \(\lambda\ge\operatorname{Lip}(f)\kappa\), then
\[
p
\le
f(x)+\lambda d_{-Z_+}(g(x))+\operatorname{Lip}(f)\varepsilon.
\]
Letting \(\varepsilon\downarrow0\) and then taking the infimum over
\(x\in\Omega\), we get
\[
p
\le
\inf_{x\in\Omega}
\bigl(f(x)+\lambda d_{-Z_+}(g(x))\bigr).
\]
Together with the opposite inequality, this proves the result.
\end{proof}

Since \(G^{-1}(0_Z)=\mathcal F\) and
\[
\operatorname{dist}(0_Z,G(x))
=
d_{-Z_+}(g(x)),
\]
the global error bound is equivalently written as
\[
\operatorname{dist}(x,G^{-1}(0_Z))
\le
\kappa\,\operatorname{dist}(0_Z,G(x)),
\qquad x\in\Omega.
\]
Thus, in the present setting, the global error bound is precisely global
metric subregularity of \(G\) at \(0_Z\) on \(\Omega\). Consequently,
Theorem~\ref{thm:exact_penalty_error_bound} may equivalently be formulated
under global metric subregularity.

The preceding results characterize exact distance penalization through norm-cone support of the natural value function and provide a verifiable sufficient condition based on global error bounds.

\subsubsection{An Infinite-Dimensional Nonconvex Conic Example}
The following example illustrates that exact distance penalization and strong norm-cone duality may hold in an infinite-dimensional conic problem with a nonconvex objective and a nonsolid ordering cone.
\begin{example}
\label{ex:conic_l2_nonconvex}
Let
\[
X=Z=\ell^2,
\qquad
Z_+:=\ell^2_+:=\{z=(z_n)_{n\ge1}\in\ell^2:\ z_n\ge0\ \forall n\}.
\]
Then \(Z_+\) is a nonempty closed convex cone. It is well known that
\(
\operatorname{int}Z_+=\varnothing,
\)
and hence \(Z_+\) is not solid.

Let \(a:=(2^{-n})_{n\ge1}\in\ell^2\). Then \(\|a\|^2=1/3\). Define
\[
\Omega:=\{x=(x_n)_{n\ge1}\in\ell^2:\ 0\le x_n\le 2a_n\ \forall n\},
\]
and
\[
f(x):=4\langle a,x\rangle-\|x\|^2,
\qquad
g(x):=a-x.
\]
The function \(f\) is not convex, since
\[
f(a)=1>
\frac{f(0)+f(2a)}{2}
=
\frac23.
\]

The conic constraint \(g(x)\in -Z_+\) is equivalent to \(x_n\ge a_n\) for all
\(n\). Hence
\[
\mathcal F
=
\{x\in\ell^2:\ a_n\le x_n\le 2a_n\ \forall n\ge1\}.
\]

\medskip
\noindent\textbf{(i) The primal problem.}
For every \(x\in\mathcal F\),
\[
f(x)-f(a)
=
\langle x-a,3a-x\rangle
=
\sum_{n\ge1}(x_n-a_n)(3a_n-x_n).
\]
Since \(a_n\le x_n\le2a_n\), every term is nonnegative. If \(x\neq a\),
then \(x_n>a_n\) for some \(n\), and the corresponding term is strictly
positive. Hence \(f(x)\ge f(a)\), with equality if and only if \(x=a\).
Therefore \(a\) is the unique minimizer on \(\mathcal F\), and
\[
\inf_{x\in\mathcal F}f(x)=f(a)=3\|a\|^2=1.
\]

\medskip
\noindent\textbf{(ii) Distance to the cone and global error bound.}
For \(u\in\ell^2\),
\[
d_{-\ell^2_+}(u)=\|u_+\|,
\qquad
u_+:=((u_n)_+)_{n\ge1}.
\]
Hence, for every \(x\in\Omega\),
\[
d_{-Z_+}(g(x))
=
\|(a-x)_+\|.
\]
Since \(\ell^2\) is a Hilbert space and \(\mathcal F\) is nonempty, closed, and
convex, the metric projection onto \(\mathcal F\) is well defined and
single-valued. A direct coordinatewise computation yields
\[
P_{\mathcal F}(x)=x+(a-x)_+,
\qquad x\in\Omega.
\]
Therefore,
\[
\operatorname{dist}(x,\mathcal F)
=
\|(a-x)_+\|
=
d_{-Z_+}(g(x)),
\qquad \forall x\in\Omega.
\]
Thus the global error bound holds with \(\kappa=1\). Equivalently,
\(G(x):=g(x)+Z_+\) is globally metrically subregular at \(0_Z\) on \(\Omega\)
with constant \(1\).

\medskip
\noindent\textbf{(iii) Failure of the direct uniform lower Lipschitz condition.}
The direct constraint perturbation is
\[
\Phi(x,z)
:=
\begin{cases}
f(x), & x\in\Omega,\ a-x\in z-Z_+,\\
+\infty, & \text{otherwise}.
\end{cases}
\]
It does not satisfy the uniform lower Lipschitz condition. Indeed, \(0\in\Omega\) and
\[
g(0)=a\notin -Z_+,
\]
so \(\Phi(0,0_Z)=+\infty\). On the other hand, since \(g(0)=a\in a-Z_+\),
\[
\Phi(0,a)=f(0)=0.
\]
Thus no \(\alpha\ge0\) can satisfy
\[
\Phi(x,z)\ge \Phi(x,0_Z)-\alpha\|z\|,
\qquad \forall (x,z)\in X\times Z.
\]

\medskip
\noindent\textbf{(iv) Exact distance penalization.}
Fix \(x\in\Omega\), and set
\[
h:=(a-x)_+,
\qquad
y:=x+h.
\]
Then \(y\in\mathcal F\) and
\[
\|h\|=d_{-Z_+}(g(x)).
\]
Since \(y=x+h\),
\[
f(y)-f(x)
=
4\langle a,h\rangle-2\langle x,h\rangle-\|h\|^2.
\]
On the set of indices where \(h_n>0\), one has \(h_n=a_n-x_n\). Hence
\[
f(y)-f(x)
=
\sum_{x_n<a_n}(a_n-x_n)(3a_n-x_n).
\]
Since \(0\le x_n<a_n\) on this index set,
\[
0\le f(y)-f(x)
\le
3\sum_{n\ge1}a_nh_n
\le
3\|a\|\,\|h\|
=
\sqrt3\,\|h\|.
\]
Since \(y\in\mathcal F\), part \textup{(i)} yields \(f(a)\le f(y)\). Therefore,
\[
f(a)
\le
f(x)+\sqrt3\,d_{-Z_+}(g(x)),
\qquad
\forall x\in\Omega.
\]
Thus, for every \(\lambda\ge\sqrt3\),
\[
\inf_{x\in\Omega}
\bigl(f(x)+\lambda d_{-Z_+}(g(x))\bigr)
\ge
f(a).
\]
Conversely, since \(a\in\mathcal F\),
\[
d_{-Z_+}(g(a))=0,
\]
and hence
\[
\inf_{x\in\Omega}
\bigl(f(x)+\lambda d_{-Z_+}(g(x))\bigr)
\le
f(a).
\]
Therefore,
\[
\inf_{x\in\Omega}
\bigl(f(x)+\lambda d_{-Z_+}(g(x))\bigr)
=
f(a)
=
1,
\qquad
\forall\lambda\ge\sqrt3.
\]
Moreover, \(\sqrt3\) is the least exact penalty parameter. Indeed, evaluating
the penalized objective at \(x=0\) gives
\[
f(0)+\lambda d_{-Z_+}(g(0))
=
\lambda\|a\|
=
\frac{\lambda}{\sqrt3}.
\]
For exactness, this value must be at least \(f(a)=1\), which requires
\[
\lambda\ge\sqrt3.
\]
Consequently,
\[
\lambda_*=\sqrt3.
\]

\medskip
\noindent\textbf{(v) Comparison with the error-bound estimate.}
For \(x,y\in\Omega\),
\[
f(x)-f(y)
=
\langle 4a-x-y,x-y\rangle.
\]
Since \(0\le x_n,y_n\le2a_n\), one has
\(\|4a-x-y\|\le4\|a\|\), and hence
\[
|f(x)-f(y)|
\le
4\|a\|\,\|x-y\|.
\]
Moreover, for \(t\in(0,2]\),
\[
\frac{|f(ta)-f(0)|}{\|ta\|}
=
(4-t)\|a\|
\longrightarrow
4\|a\|
=
\frac4{\sqrt3}
\qquad (t\downarrow0).
\]
Thus
\[
\operatorname{Lip}(f)=\frac4{\sqrt3}.
\]
Since the global error bound holds with \(\kappa=1\),
Theorem~\ref{thm:exact_penalty_error_bound} guarantees exact penalization for
\(\lambda\ge4/\sqrt3\), whereas part~(iv) gives the sharp threshold
\(\lambda_*=\sqrt3\).

\medskip
\noindent\textbf{(vi) Strong norm-cone duality.}
For every \(\lambda\ge\sqrt3\), part~(iv) and
Theorem~\ref{thm:strong_duality_conic} yield
\[
\sup D_\lambda^{\sharp_Z}
=
\inf_{x\in\mathcal F}f(x)
=
1.
\]
Moreover, Corollary~\ref{cor:exact_penalty_pointwise_biconjugation} gives
\[
\sup D^{\sharp_Z}=1.
\]
Thus both the penalized and the direct norm-cone dual problems have the primal
optimal value, although the direct perturbation fails the uniform lower
Lipschitz condition by part~(iii).

This example shows that the norm-cone duality framework applies in infinite
dimension without convexity of the objective function and without solidity of
the ordering cone. The conclusion is driven by the metric structure of the
constraint, encoded by the identity
\[
\operatorname{dist}(x,\mathcal F)=d_{-Z_+}(g(x)).
\]
\end{example}

\section{Conclusions and Future Research}
\label{sec:conclusions_future_research}

We have developed a norm-cone conjugation framework for nonconvex
optimization generated by the metric elementary family
\[
x\longmapsto r-\alpha\|x-x_0\|,
\qquad
r\in\mathbb R,\quad \alpha\geq0,\quad x_0\in X.
\]
Although $\mathcal{NC}^-$ is a proper subclass of $\mathcal{LC}^-$, both
families generate the same class of abstract convex functions. The gain is
therefore structural rather than representational: norm-cones provide an
explicit geometry in terms of centres and slopes. This leads to computable
conjugates, a geometric biconjugation formula, and a norm-cone
subdifferential whose minimal element records the least slope of a radial
support.

On $\mathcal F_{\rm nc}(X)$, this radial reduction preserves the weak
biconjugate, but it does not identify the first conjugates, the dual
parameterizations, or their sets of optimal solutions.

Its relevance for optimization is therefore not an equivalence of the
parameterized dual problems, but a reduction of the representation and,
crucially, of the effective dual parameterization.

For perturbation problems, partial norm-cone conjugation yields a radial dual
scheme. For each slope, the complete biconjugate expression at the nominal
perturbation is maximized by choosing the nominal centre. Hence the centre
can be eliminated without changing the optimal dual value, which is
\[
v_\Phi^{\sharp\sharp}(0_Z).
\]
Strong duality is therefore equivalent to exact norm-cone biconjugation of
the value function at the nominal perturbation.

A notable consequence of this reduction is that the effective norm-cone dual
is a scalar program whose only dual variable is the real slope
$\alpha\geq0$, even when the primal and perturbation spaces are
vector-valued, or possibly infinite-dimensional. Thus the dimensionality of
these spaces does not enlarge the effective dual parameterization, although
the evaluation of the scalar dual objective may still involve a
high-dimensional optimization problem.
This dimensional reduction is the main feature of the framework from
the viewpoint of nonconvex duality.

Uniform lower Lipschitz estimates provide a direct sufficient condition by
generating a norm-cone support at the nominal perturbation.

For nonconvex conic inequality problems, exact distance penalization admits
the precise support interpretation
\[
\{\lambda\geq0:\lambda\ \text{is an exact penalty parameter}\}
=
\partial^{\sharp}v_\Phi(0_Z),
\]
and the least exact parameter is
\[
|\partial^{\sharp}v_\Phi|(0_Z).
\]
Global error bounds, equivalently global metric subregularity of the
constraint mapping, provide verifiable sufficient conditions for exactness,
which in turn yields strong norm-cone duality for both the direct and the
distance-penalized perturbation schemes. The infinite-dimensional example
in Section~\ref{sec:distance_penalization_norm_cone_duality} shows that these conclusions remain applicable with a
nonconvex objective and a nonsolid ordering cone.

Thus the exact-penalty characterization supplies quantitative
information that is not contained in a zero-duality-gap statement alone:
it identifies the entire interval of exact penalty slopes and its sharp
threshold through the metric support geometry of the value function.

The results suggest several directions for further research.

From the viewpoint of nonconvex and global optimization, a natural
question is whether the scalar structure of the dual problem can be
exploited computationally, in particular for classes of problems in which
classical linear-support duality leads to higher-dimensional multiplier
spaces.
From the variational-analysis viewpoint, local versions of the present
support, duality, and penalization results, together with dual attainment and
stability questions, appear particularly natural.
It would also be useful to investigate conditions yielding sharp
estimates of the least exact penalty slope beyond those obtained from global
error bounds.

On the generalized-convexity side, maximal norm-cone minorants,
regular norm-cone convexity, and calculus rules for the norm-cone
subdifferential and its minimal support slope remain natural topics for
further study.

\section*{Funding}

The authors acknowledge financial support from the following projects:
PID2021-\allowbreak 122126NB-\allowbreak C32, funded by MICIU/\allowbreak AEI/\allowbreak 10.13039/\allowbreak 501100011033 and by
FEDER, ``A way of making Europe''; and PID2025-\allowbreak 168246NB-\allowbreak I00, funded by
MICIU/\allowbreak AEI/\allowbreak 10.13039/\allowbreak 501100011033 and by ERDF/EU.

\bibliographystyle{plainnat}
\bibliography{references}

\end{document}